\documentclass[11pt,reqno]{amsart}
\usepackage{graphicx} 
\usepackage{amsthm}
\usepackage{amsmath}
\usepackage{amssymb}
\usepackage{amsfonts}
\usepackage{mathtools}
\usepackage{fullpage}
\usepackage{python}
\usepackage{enumitem}
\usepackage{bm}
\usepackage{tikz}
\usepackage{algpseudocode}
\usepackage[linesnumbered,ruled]{algorithm2e}
\usepackage{appendix}
\usepackage{tabularx}
\usepackage{multirow}
\usepackage{comment}
\usepackage{bm}
\usepackage{subcaption}
\usepackage{longtable}
\usepackage{booktabs}
\usepackage{xcolor}
\usepackage{float}

\theoremstyle{plain}\newtheorem{theorem}{Theorem}[section]
\newtheorem{remark}{Remark}[section]
\newtheorem{definition}{Definition}[section]
\newtheorem{conjecture}{Conjecture}[section]
\newtheorem{corollary}{Corollary}[section]
\newtheorem{problem}{Problem}[section]
\numberwithin{equation}{section}
\makeatletter\let\th@plain\makeatother

\usepackage[backend=biber]{biblatex}
\title{\textnormal{Mapping Mathematical Hardness: Machine-Assisted Conjecture Discovery and the Quantification of Non-Triviality}}
\author{Madhuparna Das}
\thanks{Email: md2147@cam.ac.uk\\Department of Computer Science and Technology, University of Cambridge, Cambridge CB3 0FD, UK.}
\date{}

\begin{document}

\begin{abstract}
Machine-assisted mathematical discovery has been a long-standing challenge in machine learning and artificial intelligence. In recent years, we have seen tremendous progress with generative AI, yet its contribution to automated discovery in advanced mathematical research has been limited. One of the most difficult benchmarks in this context is the Birch test, which asks whether a machine can discover truly novel and non-trivial mathematical structures without human intervention. In this work, we particularly focus on the branch of automated conjecture discovery. We use HypothesiX, an automated conjecture mining agent and analyse its generated conjectures related to the distribution of twin primes to verify the conditions of the Birch test. Furthermore, note that automated discovery is now operating at scale, but verifying its non-triviality still depends on human evaluation. We propose a benchmark to quantify the non-triviality of machine-generated conjectures using the Mahalanobis distance within an embedding cluster of selected known mathematical conjectures. We also note that this quantified benchmark can be used as an error indication signal to localise the incorrectness of a new mathematical statement, which autoformalisers fail to verify due to their limitations in proof discovery capability.
\end{abstract}

\maketitle
\tableofcontents
\section{Introduction}
\noindent
In recent years, machine learning and automated reasoning have begun to play an important role in mathematical research. Automated systems have achieved remarkable milestones, including solving IMO problems and, more recently, establishing proofs of conjectures such as Fel's conjecture~\cite{axiommath,alphaproof,aristotle}. However, our primary motivation in this work lies not in the automation of proofs but in the process of mathematical discovery itself. For mathematicians, intuition primarily guides the formulation of the right questions; it is through those questions that genuinely novel conjectures and proofs, and ultimately breakthroughs, are born. In this work, we use {\it HypothesiX}, a conjecture-mining agent designed to emulate mathematical intuition in order to uncover novel conjectures and facilitate their further exploration. 

\medskip
Recently, Bryan Birch proposed a mathematical analogue of the Turing test, intended to provide a benchmark for assessing whether a machine assisted mathematical discovery can be regarded as genuinely significant. This was later formalised by Yang Hui He and Mikhail Burtsev~\cite{yangmika}. This benchmark consists of three key conditions:
\begin{enumerate}
\item ``The discovery should be made automatically by the AI, without human intervention."
\item ``It should uncover a concrete mathematical structure."
\item ``It should be of sufficient importance to spark new research."
\end{enumerate}

\medskip
\noindent
We note that this benchmark can be interpreted in more than one way. Mathematical discovery broadly includes conjecture generation and discovery of proof techniques, and more recently their interplay. In this paper, we solely focus on conjecture generation, as HypothesiX is built specifically for this purpose. However, in order to generate new conjectures, a conjecture mining agent may also discover new definitions, which are often subjects of conjectures, or uncover new properties of known mathematical objects along the way.

Even though ChatGPT has passed the Turing test several times, it is well known that passing the Birch test is substantially more challenging, and the current LLMs or reasoning models remain limited in this capacity. In this work, we analyse one of the conjectures generated by HypothesiX and examine how it satisfies the three conditions of the Birch test. In HypothesiX architecture, we employ GPT-5 model from OpenAI in combination with our reasoning layer to generate 78 conjectures in the domain of analytic number theory, specifically on the distribution of primes. We provide numerical evidence in favour of some generated conjectures \footnote{The codes are available on GitHub. Link: \url{https://github.com/MadhuparnaDas96/HypothesiX-Benchmark}}. Owing to the limited formalisation currently available in Lean and the complexity and the advanced nature of the generated statements, formal verification within the Lean environment is not straightforward.

It is worth noting that automated theorem provers such as Lean and Isabelle remain limited in their ability to verify truly novel conjectures, as they rely heavily on formalised mathematical libraries that are largely constructed and maintained through manual effort. Consequently, assessing the novelty and significance of newly generated conjectures often still requires human evaluation. In this work, we also propose a quantitative approach to assessing the non-triviality condition of the Birch test benchmark. Using the Mahalanobis distance, we construct an embedding model of some known conjectures in the literature and compute a score for machine-generated conjectures relative to this distribution. The methodology is described in more detail in Section~\ref{benchmark}.

\subsection{Related results}

In this section, we discuss existing architectures and methodologies for automated conjecture discovery.

Chuharski et al.~\cite{chuharski2024mining} present a ``guess-and-check" pipeline in which large language models, specifically GPT-4, Claude Sonnet, and Gemini are prompted to generate conjectures about soluble groups along with corresponding GAP code. Each conjecture is then tested in GAP, and any counterexamples identified are fed back into the prompt to refine subsequent generations. From 792 generated conjectures, approximately 420 were unique (55.48\%), indicating a substantial degree of redundancy in the outputs. Of these unique conjectures, only 40 (9.52\%) produced no counterexamples under the performed tests. The authors also note that a small number of the generated conjectures corresponded to results already present in the literature and were therefore classified as redundant.

\medskip
\noindent
On the other hand, Alhessi et al~\cite{alhessi2025_lemmanaid} embed GPT-4 within the Isabelle/HOL proof assistant~\cite{nipkow2002isabelle} to suggest lemmas for user supplied functional programs. GPT-4 proposes candidate equations, which Isabelle then tries to prove or refute. While the model occasionally hallucinates familiar patterns (e.g. trivial associativity), it also uncovers non obvious distributive laws and analogies that classic tools like QuickSpec may miss. This work underscores how LLM driven generation, coupled with formal checking, can enrich automated theory exploration.

\medskip
\noindent
Onda et al~\cite{onda2025_leanconjecturer} introduce LeanConjecturer, a pipeline that automatically generates formal conjectures in the Lean theorem prover using large language models and automated filtering. The system extracts mathematical metadata from the mathlib, prompts a language model to propose new theorem statements, and then filters them by checking syntactic correctness and removing trivial results that can be automatically proved by tactics such as \texttt{aesop}. Using 40 mathlib files as seeds, the authors generated over 12,000 candidate conjectures in the area of topology and set theory, of which about 3,700 were valid and non-trivial, averaging roughly 103 conjectures per file.

\medskip
\noindent
Although these works are an important milestone for AI in mathematics, their system does not produce results that are genuinely novel or have the potential to drive significant further research. Often, the generated conjectures can be immediately reduced to known problems, meaning the last two criteria of the Birch test are not met. The key limitations of LLM-based architectures arise from their tendency to hallucinate and to generate known structures or refinement of notion, rather than discovering truly new mathematical problems. Our reasoning layer addresses this problem and helps make progress towards passing the Birch test.

\medskip
\noindent
Note that in~\cite{chuharski2024mining}, the authors compare generated conjectures with existing results using cosine similarity, focusing primarily on known mathematical and linguistic structures. Davies et al.~\cite{alex} employed machine learning to discover new conjectures in knot theory and representation theory by training models to identify structural relationships between mathematical invariants, with significance evaluated through expert mathematician review. 

\medskip
While subsequent systems such as AlphaProof and Aletheia have surpassed this work in terms of proof capability,~\cite{alex} remains methodologically significant in the context of conjecture mining. In particular, it is one of the few approaches where the quality of generated conjectures was assessed based on their mathematical significance rather than through a purely computational filter. However, this evaluation was entirely qualitative and reliant on human judgment, lacking a replicable scoring framework that could be systematically applied across large volumes of generated conjectures.

Given the rapidly expanding landscape of AI-generated mathematical discoveries and the absence of any replicable, quantitative framework for assessing their depth, there is a clear need for an automated scoring system to evaluate the non-triviality of machine-generated conjectures. To address this, we propose a quantification of the non-triviality condition of the Birch test and apply it to score the conjectures generated by HypothesiX.


\section{Analysis of the Generated Conjectures}\label{section2}
\noindent
The agent can generate conjectures in any area of mathematics, but for clarity and consistency, we focus only on additive and analytic number theory. We have used 71 Lean files in the agent's architecture from the mathlib4 library~\cite{mathlib4} and the GPT-5 model from OpenAI. 

\subsection{Background}

The most famous unsolved problem in number theory, often referred to as the “holy grail of mathematics”, is the Riemann Hypothesis (RH), posed by B. Riemann~\cite{rh} in 1859. For $s\in\mathbb{C}$, the Riemann zeta function
\begin{align*}
\zeta(s)
=\sum_{n=1}^{\infty}\frac{1}{n^s}
\end{align*}
converges for $\Re(s)>1$. 

\begin{conjecture}[Riemann Hypothesis]
All non-trivial zeros of $\zeta(s)$ lie on the line $\Re(s)=\frac{1}{2}$.
\end{conjecture}
\noindent
Even though the problem remains unsolved, it has numerous applications, including in quantum physics and cryptography. In number theory, the significance of the Riemann Hypothesis lies in its deep connection to the distribution of prime numbers.
 
\begin{theorem}[Prime Number Theorem]
Let $\pi(x)$ denote the prime counting function
\begin{align*}
\pi(x)
=\#\{p\leq x: p\in\mathbb{P}\},
\end{align*}
where $\mathbb{P}$ denotes the set of prime numbers. Then, one has that
\begin{align*}
\lim_{x\to\infty}\frac{\pi(x)\log x}{x}=1.    
\end{align*}
\end{theorem}
\noindent
For the logarithmic integral $\mathrm{Li}(x)=\int_{2}^{x}\frac{dt}{\log t}$, the following statement is known to be equivalent to the Riemann Hypothesis~\cite{vankoch}.
\begin{align*}
\pi(x)=\mathrm{Li}(x)+O(\sqrt{x}\log x).    
\end{align*}
The above relation indicates that, under RH the deviations of $\pi(x)$ from its expected growth are tightly bounded, giving a precise measure of how irregularly primes are distributed. In 1923, Hardy and Littlewood posed a generalisation of the prime number theorem\footnote{For a detailed discussion on this topic, we refer the reader to~\cite[\S2.3]{iwanieckowlaski} and~\cite{sarnak2004}.
}.

\begin{conjecture}[First Hardy-Littlewood conjecture~\cite{hardylittlewood}]\label{hardylittewood}
Let $h_1,h_2,\ldots,h_k$ be distinct positive even integers such that the numbers of the sequence $\mathbb{P}_k=(p,p+h_1,\ldots,p+h_k)$ do not form a complete residue class with respect to any prime and let
\begin{align}\label{generalisedprimecountingfunction}
\pi_k(x)=\#\{p\leq x: p,p+h_1,\ldots,p+h_k\in\mathbb{P}\}.   
\end{align}
Then
\begin{align*}
\pi_k(x)
\sim C_k\int_{2}^{x}\frac{dt}{\log^{k+1}t},
\end{align*}
where
\begin{align}\label{ktupleconstant}
C_k
=2^k\prod_{\substack{p\in\mathbb{P}\\p\;\mathrm{odd}}}\frac{1-\frac{w(p;h_1,\ldots,h_k)}{q}}{\left(1-\frac{1}{p}\right)^{k+1}},
\end{align}
and $w(p;h_1,\ldots,h_k)$ denotes the number of distinct residues of $0,h_1,\ldots,h_k$ mod $p$.
\end{conjecture}

\medskip
\noindent
If we choose $k=1$ and $h_1=2$, then the above conjecture reduces to the twin prime conjecture~\cite{James1}, where 
\begin{align}\label{twinprimeconstant}
C_2=\prod_{\substack{p\in\mathbb{P}\\p\;\mathrm{odd}}}\frac{p(p-2)}{(p-1)^2}\approx0.66,   
\end{align}
denotes the twin prime constant. 


In the following section, we analyse one particular response by HypothesiX and show how it satisfies the conditions of the Birch test benchmark. The following results propose new structures.

\subsection{The automaticity condition}\label{selectedinequalities}

To bypass the first condition, we have used a simple prompt ``{\it Can you generate an inequality between the prime counting function and the twin prime counting function. Something simple yet non-trivial.}" 

\medskip
Note that we do not ask for any specific mathematical formulation, nor do we ask the machine to generate a new function. The prompt merely specifies a general topic rather than a precise mathematical objective.

In response to this prompt, the agent defined a new function, which is unknown in the literature. Then the agent proposed basic and ``simple" properties of the defined function. In this context, it is worth mentioning that, to build a new problem or proof structure in mathematics, mathematicians often define new functions or assemble known phenomena in a specific way to solve a problem. They then study their basic properties and eventually use them to achieve more advanced results. Without any additional training data or precise instructions, HypothesiX behaves in such a manner automatically\footnotetext{Throughout this paper, we have identified the definitions and conjectures generated by HypothesiX. We have proven some of them.}.

\begin{definition}[HypothesiX]\label{bqdefinitionst}
For a squarefree $Q$ where $6 \mid Q$, let $U_Q$ be the set of residues $r\bmod Q$ such that both $r$ and $r+2$ are coprime to $Q$. The residue-pairing bound is
\begin{align}\label{bqdefinition}
B_Q(x)=\sum_{r \in U_Q} \min\left( \pi(x; Q, r), \pi(x; Q, r+2) \right)    
\end{align}
where 
\begin{align*}
\pi(x; Q, a)
=\#\{p\leq x:p\equiv a \bmod Q\}.
\end{align*}  
\end{definition}
\noindent
From the definition of $B_Q(x)$, we can say that it is well-defined. This addresses the first half of the second condition of the Birch test,in which one is asked to discover a concrete mathematical structure.

\subsection{The non-triviality condition}

In this section, we study whether the newly defined object constitutes a genuinely novel structure, with the aim of sparking further research, thereby bypassing the second half of the second condition and the third condition of the Birch test.

Note that the concept of definition of the set $U_Q$ is already known in the literature~\cite[\S9]{HalberstamRichert} and has been used several times in numerous publications. However, the definition $B_Q(x)$ is an unknown function, which has significance, as discussed in later sections.

\begin{figure}
\centering
\includegraphics[width=0.5\linewidth]{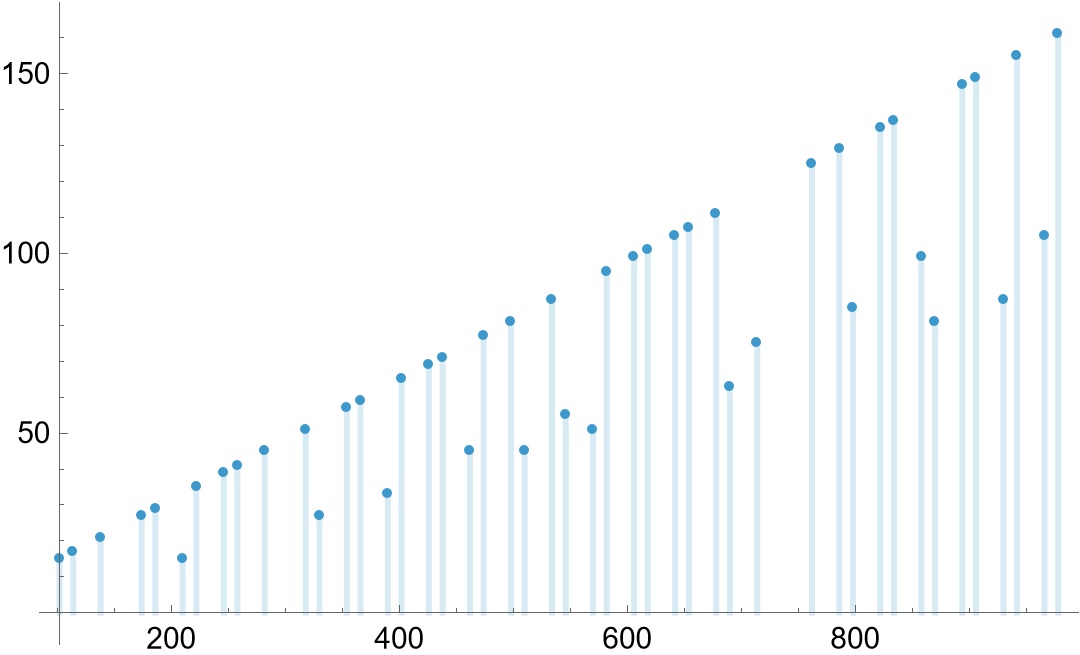}
\caption{The residue-pairing bound $B_Q(x)$ at $Q=6$}
\label{fig:placeholder}
\end{figure}

\medskip
\noindent
The following conjectures were generated by the agent, and they only study the basic properties of $B_Q(x)$ and some of them can be proved easily. However, the following results will help us show how we can define a new sieve method by studying the function $B_Q(x)$.

\begin{conjecture}[HypothesiX, Conjecture A.1]\label{a1}
For all $x \geq 7$ and any squarefree $Q$ with $6 \mid Q$:$$\pi_2(x) \leq B_Q(x)+2.$$    
\end{conjecture}
\noindent
We have numerically verified the above conjecture for $Q=30,210$ and $7\leq x\leq 10^6$, and found that it holds in all cases tested (see the GitHub repository for the code). Current methods fall short of proving the above inequality unconditionally. However, we prove a weaker version of the proposed conjecture below.

\begin{theorem}\label{pi2bq}
For all $x \geq 7$ and any square-free $Q$ with $6 \mid Q$
\begin{align*}
\pi_2(x) \ll_Q B_Q(x).    
\end{align*}
\end{theorem}

\begin{proof}
Consider the sets
\begin{align*}
S_1
=\{p\in\pi_2(x): p\leq Q\} \;\text{and}\; 
S_2
=\{p\in\pi_2(x): p> Q\}.
\end{align*}
\noindent
Since $Q$ is fixed and sqaurefree, the set $S_1$ becomes a constant, say $c_Q$. Additionally, for $p>Q$ and $p,p+2\in\mathbb{P}$, then $(p,Q)=(p+2,Q)=1$. Define
\begin{align*}
t_r(x)
\coloneqq\{p\leq x: p\in\mathbb{P}_2,\;p>Q\;\text{and}\;p\equiv r(\bmod{Q})\}.
\end{align*}
\noindent
Note that by definition $|t_r(x)|\leq\pi(x;Q,r)$. Let $\rho:p\mapsto p+2$. Then $\rho$ is injective on $t_r(x)$, and each image $p+2$ is a prime with $p+2\equiv r+2(\bmod Q)$. Thus, $|t_r(x)|\leq\pi(x;Q,r+2)+2$. Combining these, we may write
\begin{align*}
|t_r(x)|
\leq\min(\pi(x;Q,r),\pi(x;Q,r+2))+2.
\end{align*}
\noindent
Summing over the residues $r\in U_Q$, we arrive at
\begin{align*}
|S_2|\leq \sum_{r\in U_Q}|t_r(x)|
\leq B_Q(x)+2|U_Q|
\end{align*}
Therefore,
\begin{align*}
\pi_2(x)
=|S_1|+|S_2|
\leq B_Q(x)+c_Q+2|U_Q|
\ll_Q B_Q(x),
\end{align*}
completing the proof.
\end{proof}
\noindent
The above theorem states that if the twin prime conjecture is true, then $B_Q(x)$ is unbounded as $x\to\infty$ for any squarefree $Q$ with $6|Q$\footnotetext{Theorem~\ref{pi2bq} is a weaker version of the original conjecture, as the original statement is impossible to prove within current knowledge on the distribution of twin primes.}.

\begin{theorem}[HypothesiX, Conjecture A.3]\label{monotonerefinement}
If $Q_1 \mid Q_2$ are both square-free and divisible by 6, then for every $x \geq 7$
\begin{align*}
B_{Q_2}(x) \leq B_{Q_1}(x).    
\end{align*}
\end{theorem}

\begin{proof}
Let $r \in U_{Q_2}$. Thus, by Definition~\ref{bqdefinition}, $(r,Q_2) = (r+2,Q_2) = 1$. Since $Q_1$ and $Q_2$ are square-free, every prime divisor $\ell\mid Q_1$ also divides $Q_2$, i.e., $\ell \mid Q_2$. Hence $(r,Q_1) = (r,Q_2) = 1$. Therefore, the residue $\bar{r}=r\bmod{Q_1}$ lies in $U_{Q_1}$. In particular, the reduction map $\rho: r\mapsto r\bmod{Q_1}$ for $U_{Q_2}\to U_{Q_1}$ is well defined.

\medskip
Now, for $s \in U_{Q_1}$, define
\begin{align*}
v_s
\coloneqq \{r \in U_{Q_2} : r \equiv s \bmod{Q_1}\}.
\end{align*}
It is evident that $\rho$ is surjective and that the fibres $v_s$ are pairwise disjoint and cover $U_{Q_2}$. Thus,
\begin{align*}
U_{Q_2}
= \bigcup_{s \in U_{Q_1}} v_s .
\end{align*}

\noindent
Now, for $r \in v_s$, every prime $p \equiv r(\bmod{Q_2})$ satisfies $p \equiv s(\bmod{Q_1})$ because $Q_1 \mid Q_2$. Thus, the residue classes $(r \bmod{Q_2})_{r \in v_s}$ are pairwise disjoint subprogressions, all contained in the single progression $s(\bmod{Q_1})$, giving
\begin{align*}
\sum_{r \in v_s} \pi(x;Q_2,r)
\leq \pi(x;Q_1,s).
\end{align*}

\noindent
By a similar argument, we have $r+2 \equiv s+2(\bmod{Q_1})$, and hence
\begin{align*}
\sum_{r \in v_s} \pi(x;Q_2,r+2)
\leq \pi(x;Q_1,s+2).
\end{align*}

\noindent
Since $\min(a,b) \leq a$ and $\min(a,b) \leq b$, taking the minimum of these upper bounds gives
\begin{align*}
\sum_{r \in v_s} \min \bigl(\pi(x;Q_2,r),\pi(x;Q_2,r+2)\bigr)
\leq \min\bigl(\pi(x;Q_1,s),\pi(x;Q_1,s+2)\bigr).
\end{align*}

\noindent
Using the above inequality yields
\begin{align*}
B_{Q_2}(x)
&= \sum_{r \in U_{Q_2}} \min\bigl(\pi(x;Q_2,r),\pi(x;Q_2,r+2)\bigr) \\
&= \sum_{s \in U_{Q_1}} \sum_{r \in v_s} \min\bigl(\pi(x;Q_2,r),\pi(x;Q_2,r+2)\bigr) \\
&\leq \sum_{s \in U_{Q_1}} \min\bigl(\pi(x;Q_1,s),\pi(x;Q_1,s+2)\bigr)
=B_{Q_1}(x),
\end{align*}
concluding the proof.
\end{proof}

\noindent
An immediate consequence of the preceding theorem is the following.
\begin{corollary}
Let $Q$ be a fixed square-free integer divisible by $6$. Then, for every $x \geq 7$,
\begin{align*}
B_Q(x) \leq B_6(x).
\end{align*}
\end{corollary}

\begin{theorem}[HypothesiX, Conjecture A.8]\label{primedrop}
Fix squarefree $Q$ and prime $p \geq 5$ coprime to $Q$. There exists $X_0$ such that for all $x \geq X_0$
\begin{align*}
B_{pQ}(x)\leq B_Q(x)-1    
\end{align*}          
\end{theorem} 

\begin{proof}
Choose any $s_o\in U_Q$. Then by the Chinese remainder theorem each $s\in U_Q$, lifts exactly $p$ distinct classes modulo $pQ$. The fiber over $s_0$ is
\begin{align*}
v_{s_0}
=\{r\in U_{pQ}:r\equiv s_0(\bmod{Q})\},
\end{align*}
consisting of the $p-2$ lifts satisfying $r\not\equiv0,-2(\bmod{p})$. Now, we define
\begin{align*}
r_0
\coloneqq F(s_0,Q,p-2,p),
\end{align*}
where $F(s_0,Q,p-2,p)$ means the unique residue class $r_0(\bmod{pQ})$ such that
\begin{align*}
r_0\equiv s_0(\bmod{Q})\quad\text{and}\quad r_0\equiv p-2(\bmod{p}).    
\end{align*}
Similarly, 
\begin{align*}
r_3
\coloneqq F(s_0+2,Q;2,p).
\end{align*}

\noindent
Note that $p\geq5$ and $(p,Q)=1$. Then, for $r_0\equiv p-2(\bmod{p})$, $(r_0,p)=1$ and $(r_0,Q)=(s_0,Q)=1$. Analogously, $(r_3,p)=(r_3,Q)=(s_0+2,Q)=1$. By Dirichlet's theorem, we can write
\begin{align*}
\pi(p,Q;r_0)\geq 1\quad\text{and}\quad
\pi(p,Q;r_3)\geq 1.
\end{align*}

\noindent
Since $r_0\equiv p-2\equiv -2(\bmod{p})$, we have $p|(r_0+2)$, so $r_0\notin U_{pQ}$ and thus $r_0\notin v_{s_0}$. Therefore, $r_0$ contributes to $\pi(x;Q,s_0)$ but not to $\sum_{r\in v_{s_0}}\pi(x;pQ,r)$. For any $r\in v_{s_0}$, we have $r \not\equiv -2(\bmod{p})$ and $r\not\equiv 0(\bmod{p})$, 
so $r+2\not\equiv 2(\bmod{p})$. But $r_3\equiv 2(\bmod{p})$, so $r_3$ is not of the form $r+2$ for any $r\in v_{s_0}$. Since $r_3 \equiv s_0+2(\bmod{Q})$, it does contribute to $\pi(x,;Q,s_0+2)$ but not to $\sum_{r\in v_{s_0}}\pi(x;pQ,r+2)$. 

\medskip
Set
\begin{align*}
A
\coloneqq\sum_{r\in v_{s_0}}\pi(x;pQ,r)\quad\text{and}\quad
B
\coloneqq\sum_{r\in v_{s_0}}\pi(x;pQ,r+2).
\end{align*}
Then for $x\geq X_0$, we have
\begin{align*}
\pi(x;Q,s_0)
\geq A +\pi(x;pQ,r_0)\geq A+1.
\end{align*}
Similarly, $\pi(x;Q,s_0+2)\geq B+1$. Since the $\min$ function is monotone increasing
\begin{align*}
\min\bigl(\pi(x; Q, s_0),\pi(x; Q, s_0 + 2)\bigr)
\geq\min(A+1, B+1) 
=\min(A,B)+1.    
\end{align*}
By the inequality $\sum_{r}\min(a_r,b_r)\leq\min(\sum_{r}a_r,\sum_{r}b_r)$ for $a_r,b_r\geq 1$, yields
\begin{align*}
\min\bigl(\pi(x;Q,s_0),\pi(x; Q, s_0 + 2)\bigr)
\geq\sum_{r\in v_{s_0}}\min\bigl(\pi(x;pQ,r),\pi(x;pQ,r+2)\bigr)+1.    
\end{align*}
Summing over all $s\in U_Q$ and using non-negativity of each term from Theorem~\ref{monotonerefinement}, 
\begin{align*}
B_Q(x)-B_{pQ}(x)
=&\sum_{s\in U_Q}\min\bigl(\pi(x;Q,s),\pi(x;Q,s+2)\bigr)\\
&-\sum_{s\in U_Q}\sum_{r\in v_s} \min\bigl(\pi(x;pQ,r),\pi(x;pQ,r+2)\bigr)
\ge 1,    
\end{align*}
thereby concluding the proof.
\end{proof}

\noindent
Before exploring any further generated conjectures, we explain how $B_Q(x)$ is connected to known problems in the literature and how its properties can be used to develop further methodologies in sieve theory. It is worth noting that using the preceding discussion, our aim is to show that the generated output by HypothesiX is truly novel and sufficient importance to spark new research.   

\medskip
\noindent
Sieve theory begins with Eratosthenes' ancient procedure of systematically eliminating multiples of small primes to isolate primes up to a bound. The first modern breakthrough came from Brun in 1920, who truncated the inclusion-exclusion process to obtain finite, usable bounds proving, among other things, that the sum of reciprocals of twin primes i.e., $\sum_{p\in\mathbb{P}_2}\left(\frac{1}{p}+\frac{1}{p+2}\right)$ converges. In 1947, Selberg then replaced Brun's combinatorial truncation with an optimised quadratic weighting scheme, yielding sharper upper bounds and underpinning results like Chen's theorem~\cite{chen} that every large even number equals $p+P_2$, where $P_2$ is a semiprime\footnote{A semiprime is the product of two prime numbers. For a comprehensive reading on sieve methods, we refer the reader to~\cite{harman}.}. Running parallel to these was the large sieve of Linnik and Bombieri-Vinogradov, which gave powerful average distribution results for primes in arithmetic progressions. The field was then transformed by Goldston-Pintz-Y{\i}ld{\i}r{\i}m~\cite{GPY1,GPY2,GPY3,GPY4}, and culminating later in a breakthrough result, Maynard~\cite{James2} and Tao introduced a multidimensional sieve, optimising weights over $k$-tuples of shifts simultaneously to prove that prime gaps are bounded infinitely often, yielding $\liminf_n(p_{n+1}-p_n) \leq 12$ under the Elliott-Halberstam conjecture. Threading through all of these methods, the unresolved parity problem is still the fundamental barrier preventing a proof of the twin prime conjecture or Goldbach's conjecture.

\begin{conjecture}[Parity problem~\cite{tao}]
If $A$ is a set whose elements are all products of an odd number of primes (or are all products of an even number of primes), then (without injecting additional ingredients), sieve theory is unable to provide non-trivial lower bounds on the size of $A$. Also, any upper bounds must be off from the truth by a factor of 2 or more.  
\end{conjecture}

\noindent
Leveraging the nonlinear and envelope structure of $B_Q(x)$, we propose a hybrid sieve approach that integrates the Maynard–Tao sieve, and outline how it may improve our understanding of the distribution of twin primes in arithmetic progressions beyond the existing literature. Furthermore, we present four open problems in this direction that are not machine-generated, but instead arise from our exploration of this new hybrid sieve. Finally, we formulate the parity problem in a quantitative manner for the first time within the framework of $B_Q(x)$ and its properties.

\medskip

The purpose of the following argument is to show that the machine-generated definition of $B_Q(x)$, along with related conjectures, can spark new research and is of sufficient importance to make further progress in the known literature on unsolved problems.

\subsubsection{The hybrid sieve setup}\label{hybridsieve} 

Although the Maynard-Tao sieve is multidimensional, it was primarily developed to study lower bound for gaps between primes, motivated by Selberg sieve. 

\medskip
\noindent
For fixed square-free $Q$ s.t. $6|Q$, we define the vector
\begin{align*}
V_{Q}(x):=(\pi(x;Q,r))\in (\mathbb{Z}/Q\mathbb{Z})^{\times} \in \mathbb{R}^{\phi(Q)}.
\end{align*}
\noindent
Then the pairing function is a map such that
\begin{align*}
B_Q(x)
=\Phi_Q(V_Q(x)), \;\text{where}\; \Phi_Q(v):= \sum_{r\in U_Q}\min(v_r,v_{r+2}).
\end{align*}
\noindent
The function $\Phi_Q: \mathbb{R}^{\phi(Q)} \to \mathbb{R}_{\geq 0}$ is ``concave" because each term $\min(v_r,v_{r+2})$ is a concave function of $(v_r,v_{r+2})$, and a non-negative sum of concave function is concave. 

\medskip
\noindent
For each $r\in U_Q$, we define the minimum adaptive weight
\begin{align*}
\sigma_{Q}(r)
:=
\frac{\min(v_r,v_{r+2})}{v_r}\in (0,1].
\end{align*}
\noindent
Note that $\sigma_Q(r)=1$ precisely when class $r$ is the binding constrain or twin pair formulation and $\sigma_Q(r)<1$, whenever $r$ has more primes than can be matched by the residue class $r+2$. Thus, the pairing bound now literally 
\begin{align*}
B_Q(x)
:=\sum_{r\in U_Q}v_r \sigma_Q(r),
\end{align*}
\noindent
a prime count in admissible class discounted class by class on how tightly the pairing constrain binds. Define the pairing defect 
\begin{align*}
\Delta_{Q}(x):= cB_Q(x)-\pi_2(x)\geq 0,    
\end{align*}
for some constant $c>0$. By the monotone refinement theorem $\Delta_Q(x)$, is non-increasing in $Q$ along the primorial ladder $6|30|210|\cdots$, thus 
\begin{align*}
\Delta_6(x)\geq \Delta_{30}(x)\geq \Delta_{210}(x)\geq \cdots\geq\Delta_{\infty}(x):= \lim_{Q}\Delta_Q(x)\geq 0.
\end{align*}
\noindent
This structure measures how much the local modular structure overestimate the true twin prime count.

\medskip
\noindent
We recall the Maynard sieve weight for a prime $k$-tuple $\mathcal{H}=\{h_1,\ldots, h_k\}$ produces weight
\begin{align*}
w_n:=\left(\sum_{\substack{d_i|n+h_i\\\forall i\\ d_i<R}}\lambda_{d_1,\ldots, d_k}\right)^2\; \text{where}\; \lambda_d=\mu(d)(\log R/d)^k,    
\end{align*}
optimised so that the ratio
\begin{align*}
\rho_{\mathrm{MT}}
=\frac{\sum_{n}w_n\sum_{i=1}^{k}{\bf{1}}_{(n+h_i)\in\mathbb{P}}}{\sum_{n}w_n}
\end{align*}
is maximised. For bounded gaps, Maynard showed $\rho_{\mathrm{MT}}>m$, implies at least $m+1$ primes in infinitely many touples of $\mathcal{H}$. Since we are in the set up of the twin primes, we choose the pair $\mathcal{H}=\{0,2\}$, decompose the Maynard sum by $r(\bmod{Q})$. Thus, for each $r\in U_Q$, let 
\begin{align*}
w_n^{(r)}:= w_n{\bf{1}}_{n\equiv r(\bmod{Q})}\;\text{and}\;
\Tilde{w}_n:=\sum_{r\in U_Q} w_n^{(r)}\sigma_Q(r).
\end{align*}
Now, the modified ratio
\begin{align*}
\Tilde{S}_1
=\sum_{n}\Tilde{w}_n,
\quad \Tilde{S}_2 
=\sum_{n}\Tilde{w}_n({\bf{1}}_{n\in\mathbb{P}}+{\bf{1}}_{n+2\in\mathbb{P}}),
\end{align*}
where
\begin{align*}
S_1=\sum_{n\leq x}w_n,\quad
S_2=\sum_{n\leq x} w_n({\bf{1}}_{n\in\mathbb{P}}+{\bf{1}}_{n+2\in\mathbb{P}}).
\end{align*}
Since, $\sigma_Q(r)\in(0,1]$, we have $\Tilde{w}_n\leq w_n$, whenever $v_r>v_{r+2}$. Then we claim that $\Tilde{\rho}_Q\geq \rho_{\mathrm{MT}}$. Decomposing by class
\begin{align*}
\Tilde{\rho}_Q-\rho_{\mathrm{MT}}
=\frac{\Tilde{S_2}S_1-S_2\Tilde{S_1}}{S_1\Tilde{S_1}}.
\end{align*}
\noindent
Since,
\begin{align*}
\Tilde{S_1}
=S_1-\sum_{r}(1-\sigma_Q(r))S_1(r),\quad \Tilde{S}_2=S_2-\sum_{r}(1-\sigma_Q(r))\cdot S_2(r).
\end{align*}
\noindent
Thus, the numerator reduces to
\begin{align*}
\sum_{r}(1-\sigma_Q(r))(S_2(r)S_1-S_1(r)S_2)\frac{1}{S_1}.
\end{align*}
The factor $S_2(r)/S_1(r)$ is the local Maynard ratio in class $r$, the density of twin primes among all sieve supported $n\equiv r$. Thus, the class $r$ for which $1-\sigma_Q(r)>0$ are precisely those $v_r>v_{r+2}$ i.e., class $r$ has an excess of primes that can not be matched into twin pairs. For those classes $S_2(r)/S_1(r)<S_2/S_1-\rho_{\mathrm{MT}}$ because not all primes in class $r$ have a prime in $r+2$. Therefore, each term in the sum is positive and $\Tilde{\rho}_Q>\rho_{\mathrm{MT}}$. Recall that $\Delta_Q(x)=cB_Q(x)-\pi_2(x)$. Then
\begin{align}\label{sieveweightconnection}
\Tilde{\rho}_Q 
=\rho_{\mathrm{MT}}\left(1+\frac{\Delta_Q(x)}{cB_Q(x)-\Delta_Q(x)}\right)
=\rho_{\mathrm{MT}}\frac{cB_Q(x)}{\pi_2(x)}.
\end{align}
\noindent
We define the primorial sequence 
\begin{align*}
Q_k
\coloneqq \prod_{p\leq p_{k+2}}p.
\end{align*}
Then by~\eqref{sieveweightconnection} and Theorem~\ref{monotonerefinement}, yields
\begin{align*}
\Tilde{\rho}_{Q_{k+1}}
=\rho_{\mathrm{MT}}\frac{cB_{Q_{k+1}}(x)}{\pi_2(x)}
\leq \rho_{\mathrm{MT}}\frac{cB_{Q_{k}}(x)}{\pi_2(x)}
=\Tilde{\rho}_{Q_{k}}.
\end{align*}
Hence, $\{\Tilde{\rho}_{Q_{k}}\}$ itself is monotone and non-increasing, and
\begin{align*}
\Tilde{\rho}_{Q_{k}}-\Tilde{\rho}_{Q_{k+1}}
\geq \rho_{\mathrm{MT}}\frac{c}{\pi_2(x)}(B_{Q_{k}}(x)-B_{Q_{k+1}}(x)).
\end{align*}
By Theorem~\ref{primedrop} for all $x\geq X_0$, we have $B_{Q_{k+1}}(x)\leq B_{Q_{k}}(x)-1$ and thus
\begin{align*}
\Tilde{\rho}_{Q_{k+1}}-\Tilde{\rho}_{Q_{k}}  
\geq c\frac{\rho_{\mathrm{MT}}}{\pi_2(x)}.
\end{align*}
Therefore,
\begin{align*}
\Tilde{\rho}_{\infty}
\coloneqq\lim_{k\to\infty} \Tilde{\rho}_{Q_{k}} 
=\rho_{\mathrm{MT}} \frac{cB_{\infty}(x)}{\pi_2(x)},
\end{align*}
where $B_{\infty}(x)=\lim_{k\to\infty}B_{Q_k}(x)$.

\subsubsection{Open directions and the parity problem}\label{openproblems}
Using the hybrid sieve framework developed in the previous section, we formulate several open problems. Note that the following problems are not machine generated. 

\medskip
The key distinction of the hybrid sieve framework, in comparison to the Maynard–Tao sieve, lies in the fact that the sieve weight ratio $\tilde{\rho}_Q$ is non-increasing, monotonic and have concavity. We exploit this property to formulate the following problems, which offer new directions for research at the intersection of sieve theory and harmonic analysis.

\medskip
\noindent
Define $\phi(s)=\Tilde{S}_1(e^s)$ by interpolating $Q \mapsto \Tilde{S}_1(Q)$ along the primorial ladder via $s = \log Q$. Since $\widetilde{S}_1(Q)$ is non-increasing in $Q$ by Theorem~\ref{monotonerefinement}, and $\Phi_Q$ is concave, the function $\phi$ is a natural candidate for complete monotonicity.

\begin{problem}[Complete Monotonicity for Primorial moduli]\label{1}
For fixed $x$ and the primorial ladder $Q_k$, define the sequence $s_k:=\Tilde{S}_1(Q_k)$. Prove that the finite differences $(-1)^n \delta^n s_k \geq 0$ for all $n \geq 0$, where $\delta s_k = s_{k+1}-s_k$, and conclude that $\Tilde{S}_1$ extends to a completely monotone function of $\log Q$, admitting the Bernstein representation
\begin{align}\label{s1q}
\Tilde{S}_1(Q) 
=\int_0^\infty Q^{-t}\, d\mu(t)    
\end{align}
for a positive measure $\mu$ depending on $x$.
\end{problem}

Given the expression~\eqref{s1q}, the measure $\mu$ encodes how $\Tilde{S}_1$ decays as the primorial moduli grows. Since $\Tilde{S}_1(Q)$ depends on the twin primes in arithmetic progressions mod $Q$, the measure $\mu$ implicitly carries information about distribution of twin primes.

\begin{problem}\label{2}
Let $\mu$ be the Bernstein measure of $\Tilde{S}_1(Q)$. Prove that the $k$-th moment of $\mu$ satisfies  
\begin{align*}
\int_{0}^{\infty} t^k \mu d(t)
=(-1)^k\frac{d^k}{d(\log Q)^k}\Tilde{S}_1(Q)\bigg|_{Q=Q_0}
\end{align*}
and express these moments explicitly in terms of $\sigma_Q$. 
\end{problem}
\noindent
In light of the preceding problem, one may further explore what the Bombieri-Vinogradov theorem and the Elliott-Halberstam conjecture respectively imply about the support and concentration of $\mu$. 

From the previous argument, we know that $\Tilde{\rho}_Q>\rho_{\mathrm{MT}}$. By expression~\eqref{s1q}, sufficiently large $\Tilde{S}_1(Q)$ for fixed $Q$ corresponds to $\mu$ being more concentrated near $t=0$, since $Q^{-t}\to 1$ as $t\to 0$. 

\begin{problem}\label{3}
Show that $\Tilde{\rho}_Q>\rho_{\mathrm{MT}}$ is equivalent to $\mu$ satisfying
\begin{align*}
\int_{0}^{\varepsilon} d\mu(t)
>\int_{0}^{\varepsilon} d\mu_{\mathrm{MT}}(t), \quad\text{for all}\; \varepsilon>0,
\end{align*}
where $\mu_{\mathrm{MT}}(t)$ is the Bernstein measure of the Maynard-Tao sieve weight. The $\sigma_Q$ discounting retains only well-paired classes where $v_r\approx v_{r+2}$, suppressing classes $v_r\gg v_{r+2}$. In terms of the Bernstein measure, this is exactly concentration of $\mu$ near $t=0$, since mass at large $t$ corresponds to classes that become increasingly imbalanced under primorial refinement.
\end{problem}

We note that if $\Delta_Q\not\to 0$, then $\mu$ necessarily retains residual mass at high frequencies mass that no choice of $\sigma_Q$ discounting can remove. This offers a natural spectral reading of the parity obstruction, one that is intrinsic to the structure of $\mu$ rather than an external limitation imposed on the sieve.

\begin{problem}[Tauberian theorem for $\Tilde{S}_1(Q)$]\label{4}
Suppose
\begin{align*}
\mu([0,\varepsilon]) \sim C \varepsilon^\alpha \quad \text{as}\quad \varepsilon \to 0^{+}.    
\end{align*}
Prove that
\begin{align*}
\widetilde{S}_1(Q) \sim \frac{C'}{(\log Q)^\alpha} \quad \text{as } Q \to \infty,    
\end{align*}
with $C'$ explicitly determined by $C$ and $\alpha$, and derive a full asymptotic expansion in descending powers of $\log Q$.    
\end{problem}

\noindent
This yields two consequences that are new to the literature and spark further research. First, it provides an asymptotics for $\Tilde{S}_1$ as a function of the primorial moduli $Q$. Second, and more importantly, it establishes a direct and quantitative link between the decay rate of $\Tilde{S}_1(Q)$ and the concentration of $\mu$ near $t = 0$, which by Problem~\ref{3} is governed by the pairing defect $\Delta_Q(x)$. 

In particular, if the Tauberian asymptotics force the concentration of $\mu$ near $t = 0$ to grow without bound as $Q \to \infty$, this would imply $\Delta_Q(x) \to 0$, which is precisely the parity obstruction. Problem~\ref{4} is therefore not merely an asymptotic refinement, but a potentially viable route to the parity problem itself, approached through the concentration behaviour of $\mu$ rather than through any combinatorial sieve theoretic argument.

\begin{remark}
We note that the original conjecture proposed by HypothesiX, on the inequality concerning $\pi_2(x)$ and $B_Q(x)$, is almost as difficult as the parity problem itself. We only prove a weaker form in Theorem~\ref{pi2bq}. We have further analysed this conjecture in Section~\ref{appendix}. The proposed four problems above can be considered as an intermediate step to prove Conjecture~\ref{a1}.
\end{remark}
\noindent
Finally, we conclude this section by expressing the parity problem in terms of the pairing defect function and explaining how this reformulation can be useful in comparison to the recent formulation by Murty and Vatwani~\cite{murty}.

\begin{conjecture}[Parity problem]
There is no known sieve method that can prove that
\begin{align*}
\lim_{Q\to\infty} \Delta_Q(x)=0.    
\end{align*}    
\end{conjecture}

\noindent
Murty and Vatwani address the parity problem through the correlation of the von Mangoldt function $\Lambda(n)$ with the shifted M\"{o}bius function $\mu(n+h)$, where the Elliott-Halberstam conjecture controls the distribution of $\Lambda(n)$ in arithmetic progressions and a shifted Möbius equidistribution conjecture controls $\mu(n+h)$ over shifted primes. Neither of these objects has any sensitivity to which specific residue classes modulo $Q$ the primes occupy. Our machine generated framework instead resolves the parity obstruction into a sequence of local, class by class pairing imbalances, each measured by $\sigma_Q(r)$ and accumulating monotonically along the primorial ladder. This local resolution of the parity barrier is combinatorially new, and offers the prospect of attacking the obstruction incrementally, one primorial level at a time, rather than requiring a global conjecture on the behaviour of $\mu(n+h)$ whose depth is comparable to Chowla's conjecture.

However, their formulation carries a decisive advantage that the present framework does not yet match: they identify a precisely statable conjecture whose truth directly implies the existence of infinitely many twin primes. The condition $\Delta\to 0$ remains the natural target of the present framework, but no tractable sufficient condition for it has yet been established.

\medskip
\noindent
We conclude this section by noting that the two frameworks attack the parity barrier from different directions, each with its own natural advantages, and progress in either is likely to illuminate the other.

\subsubsection{Comparison with existing results}\label{comparsion}

The notion introduced in Sections~\ref{hybridsieve} and~\ref{openproblems} sits at the intersection of sieve theory and harmonic analysis in a way that has no direct precedent in the literature, showing that the machine generated framework $B_Q(x)$, along with its properties, sparks further research. Below, we briefly summarise its importance relative to existing work.

\begin{enumerate}
\item The monotone refinement of $B_Q(x)$ established in Theorem~\ref{monotonerefinement} extends naturally to the hybrid sieve ratio $\Tilde{\rho}_Q$, but the stronger claim of Problem~\ref{2} that is completely monotone along the primorial ladder and admits a Bernstein representation, has no analogue in any existing sieve theoretic work. While the primorial ladder appears in Halberstam-Richert and the Selberg sieve as a natural sequence of moduli, no prior work asks whether sieve weight sums along this ladder are completely monotone in the Hausdorff sense or uses Bernstein–Widder theory as a structural tool.

\medskip
\item Problem~\ref{2}, connecting the moments of the Bernstein measure $\mu$ to the local pairing weights $\sigma_Q(r)$, has no close analogue in the literature. The Bombieri–Vinogradov theorem and Elliott–Halberstam conjecture control primes in arithmetic progressions on average over moduli but say nothing about moments of any measure associated to sieve weights. Study of the moment problems in analytic number theory, such as Sato–Tate moments or moments of $L$-functions, concern entirely different objects.

\medskip
\item Problem~\ref{3} gives a third independent formulation of the parity obstruction, distinct from Selberg's combinatorial statement, the bilinear Liouville correlation of Friedlander–Iwaniec, and the shifted M\"{o}bius equidistribution of Murty–Vatwani~\cite{murty}. The latter two approaches inject external analytic information about the M\"{o}bius or Liouville function from outside the sieve. The present reformulation that the parity obstruction corresponds to residual mass of $\mu$ at large $t$ is  intrinsic to the functional analytic structure of $\Tilde{S}_1(Q)$ and requires no further information.

\medskip
\item Problem~\ref{4} applies a Karamata-type Tauberian theorem to sieve weight sums along the primorial ladder, a context in which Tauberian methods have not previously been used. The closest existing applications are the Wiener–Ikehara proof of the prime number theorem and the Selberg–Delange method for Dirichlet series, but these operate on fundamentally different objects. The conclusion of Problem~\ref{4} has no parallel in the existing literature and would constitute a new route into the parity problem.
\end{enumerate}

\section{The non-triviality Benchmark}\label{benchmark}
\noindent
In this section, we benchmark the conjectures generated by HypothesiX. In the previous section, we explored a particular output and showed how it satisfies the conditions proposed in the Birch test. However, analysing each conjecture in depth to inspire further research would require substantial time and effort. Therefore, we verify the correctness of the generated conjectures using Mathematica or manual evaluation and propose a method to quantify the non-triviality condition for the Birch test benchmark.

\subsection{The Birch Test} 

We quantify the third condition of the Birch test, commonly referred to as the non triviality criterion, and demonstrate that the conjectures generated by HypothesiX make measurable progress towards satisfying this requirement. We emphasise that this constitutes one possible approach to quantification, and that alternative or more refined measures may exist.

\subsubsection{Quantification of non-triviality}\label{nontriviality}
Our objective is to formalise this intuition by defining a quantitative notion of non-triviality grounded in structural properties of mathematical statements. 

\medskip

First we fix a mathematical domain $\mathcal{D}$, for the purpose of this article we have fixed 18 well-known unsolved problems from analytic and additive number theory. Let $\mathcal{J}$ denote the set of all conjectures expressible in $\mathcal{D}$. We define a feature map ${\bm{\theta}}: \mathcal{J}\to\mathbb{R}^d$, where ${\bm{\theta}}(c)$ represents the fundamental structural properties of a conjecture $c$. We consider
\begin{align*}
{\bm{\theta}}(c)
=(\theta_1(c),\ldots,\theta_6(c)),
\end{align*}
where:
\begin{itemize}
\item $\theta_1$ is the number of known equivalent formulations or strong consequences.
\item $\theta_2$ measures the minimum axiomatic strength and proof-theoretic complexity required to prove $c$, with higher scores assigned to conjectures that exceed the capacity of all currently known axiom systems and proof methodologies.
\item $\theta_3$ counts non-trivial known results related to the conjecture $c$ in the literature.
\item $\theta_4$ measures computational verifiability on bounded instances.
\item $\theta_5$ counts the number of known results that assume $c$ as true and build upon it non-trivially across multiple mathematical domains, independent of whether $c$ has been proved.
\item $\theta_6$ counts independent reductions of $c$ to established hard problems.
\end{itemize}

\noindent
Consider the exemplar set of conjectures $R\subset\mathcal{D}$ such that $R=\{c_1,\ldots,c_n\}$. Let
\begin{align*}
X_R
=\{{\bm{\theta}}(c_i):c_i\in R\}.
\end{align*}

\noindent
We model $X_R$ as samples from an unknown distribution $\mathcal{H}$ representing structural hardness. More precisely, the empirical hardness distribution is defined as
\begin{align*}
\mathcal{H}
\coloneqq \frac{1}{n}\sum_{i=1}^{n}\delta_{c_i},
\end{align*}
where $\delta_c$ is Dirac delta function and $n=18$. Table~\ref{knownconjectures} presents the 18 conjectures in the reference set $R$ and their assigned score vectors $\bm{\theta}(c)$\footnote{Given the size of the reference set $R$, scores are assigned on the interval [1,10]. For a larger exemplar set, this range can be scaled accordingly.}.

\medskip
\begin{table}[ht]   
\begin{center}
\begin{tabular}{|c|c|}
\hline
Conjectures &  $\bm{\theta}=$ [$\theta_1,\theta_2,\theta_3,\theta_4,\theta_5,\theta_6$]\\
\hline
Riemann Hypothesis & [10,10,10,10,10,10]\\
\hline
Twin Prime Conjecture & [7,8,8,8,6,4]\\
\hline
Prime $k$-tuples Conjecture & [9,9,9,9,7,4]\\
\hline
Elliott–Halberstam Conjecture & [8,8,7,5,7,6]\\
\hline
Generalised Elliott–Halberstam Conjecture & [8,9,6, 4,6,6]\\
\hline
Bateman–Horn Conjecture & [10,9,7,9,7,5]\\
\hline
Dickson's Conjecture & [8,7,8,8,6,4]\\
\hline
Polignac's Conjecture & [6,7,8,9,6,4]\\
\hline
Firoozbakht's Conjecture & [5,7,4,10,3,2]\\
\hline
Granville's Refinement of Cram\'{e}r's Conjecture & [4,8,5,3,3,2]\\
\hline
Hardy–Littlewood Second Conjecture & [4,6,5,3,3,2]\\
\hline
Bunyakovsky Conjecture & [7,7,5,9,4,5]\\
\hline
Chowla Conjecture & [8,9,8,6,7,6]\\
\hline
Shanks Conjecture & [4,5,4,8,4,3]\\
\hline
Schinzel's Hypothesis & [9,9,6,7,7,6]\\
\hline
Legendre's Conjecture & [5,5,6,9,6,4]\\
\hline
Landau's Fourth Problem & [6,6,7,9,5,5]\\
\hline
Parity Problem & [7,10,8,5,7,7]\\
\hline
\end{tabular} 
\end{center}
\caption{Score assignments for the examplar set $R$}
\label{knownconjectures}
\end{table}

\begin{remark}
All conjectures in $R$ concern the distribution of primes and are closely related to the Riemann Hypothesis, which is the same mathematical territory as the conjectures we generate and evaluate in this work. This ensures that the hardness scores we assign are meaningful and comparable across the examplar set and the generated conjectures.    
\end{remark}

\subsubsection{Non-triviality measure by Mahalanobis distance}\label{mahalanobis}

Standard Euclidean distance assumes that the data, in this case the motif of a conjecture, are distributed spherically. However, the complexity of mathematical conjectures is often skewed. For example, a small increase in logical depth may be far more significant than a large increase in the number of equivalent formulations. To account for such differences in scale and correlation among structural components, we employ the Mahalanobis distance.

\medskip
\noindent
Let $\mu\in\mathbb{R}^d$ and $\Sigma\in\mathbb{R}^{d\times d}$ denote the empirical mean and covariance of $X_R$. To ensure invertibility, the covariance is regularised as $\Sigma \leftarrow \Sigma + \varepsilon I$ where $\varepsilon = 10^{-6}$ and $I$ is the $6\times 6$ identity matrix. For any conjecture $c\in\mathcal{J}$ with feature vector ${\bm{\theta}}={\bm{\theta}}(c)$, the squared Mahalanobis distance is defined as\footnote{Note that one may choose $\mathcal{J}$ to be a large set , here we have chosen $\mathcal{J}=R$.}
\begin{align}\label{d2formula}
d^2({\bm{\theta}}) 
=({\bm{\theta}}-\mu)^{\top}\Sigma^{-1}({\bm{\theta}}-\mu).
\end{align}
This quantity measures the deviation of $c$ from the exemplar set after accounting for feature correlations. For a reference conjecture ${\bm{\theta}}^{(i)}={\bm{\theta}}(c_i) \in R$, the non-triviality score is computed via a leave one out procedure. Let $R_{-i} = R \setminus \{c_i\}$ denote the reference set with $c_i$ excluded. Then we define
\begin{align}\label{upsilon}
\Upsilon({\bm{\theta}}^{(i)}) 
\coloneqq \frac{\#\{{\bm{\theta}}^{(i)} \in R_{-i} : d^2({\bm{\theta}}^{(j)}) \leq d^2({\bm{\theta}}^{(i)})\}}{n-1}.
\end{align}
The scale is calibrated so that $\Upsilon = 1$ for the ceiling set $\mathcal{C} = \{$Riemann Hypothesis$\}$, which defines the upper boundary of structural hardness in $R$. For the ceiling conjecture, the comparison pool is restricted to $R \setminus \mathcal{C}$, ensuring $\Upsilon = 1$ holds by construction.

For a new conjecture $c_{\mathrm{new}}\notin R$, the feature vector $\bm{\hat{\theta}}(c_{\mathrm{new}})=[\hat{\theta}_1(c_{\mathrm{new}}),\ldots,\hat{\theta}_6(c_{\mathrm{new}})]$ is not directly observed. We estimate it via a softmax-weighted interpolation over the examplar set $R$, using a hybrid similarity measure that combines semantic content with structural motif. Let $e(c_{\mathrm{new}})$ and $e(c_i)$ denote the $\ell_2$-normalised text embeddings of $c_{\mathrm{new}}$ and $c_i$ respectively, and let $\hat{\theta}_{\text{norm}}(c_{\mathrm{new}})$ and $\theta_{\text{norm}}(c_i)$ denote their $\ell_2$-normalised theta vectors. The hybrid representation is defined as
\begin{align*}
\tilde{e}(c_{\mathrm{new}}) = \ell_2\!\left[(1-\lambda)\cdot e(c_{\mathrm{new}}) \;\Big\|\; \lambda\cdot\hat{\theta}_{\text{norm}}(c_{\mathrm{new}})\right], \qquad \lambda = 0.7,
\end{align*}
where $\|$ denotes vector concatenation and $\lambda$ controls the relative contribution of structural motif over linguistic content. Since $\bm{\hat{\theta}}(c_{\mathrm{new}})$ is not observed directly, a two-pass procedure is employed. 

In the first pass, a simple estimate $\bm{\hat{\theta}}(c)$ is obtained via text-only softmax similarity. In the second pass, the hybrid representation is constructed using $\bm{\hat{\theta}}(c)$ and the final softmax weights are computed as
\begin{align*}
w_i = \frac{\exp\!\left(\langle \tilde{e}(c_{\mathrm{new}}),\, \tilde{e}(c_i)\rangle / \tau\right)}{\sum_{j=1}^{n} \exp\!\left(\langle \tilde{e}(c_{\mathrm{new}}),\, \tilde{e}(c_j)\rangle / \tau\right)}, \qquad \tau = 0.01,
\end{align*}
where $\langle\cdot,\cdot\rangle$ denotes cosine similarity and $\tau$ is a temperature parameter controlling the sharpness of the weight distribution. The final estimated feature vector and non-triviality score are then
\begin{align}\label{thetahatdefinition}
{\bm{\hat{\theta}}}(c_{\mathrm{new}}) 
=\sum_{i=1}^{n} w_i \cdot \bm{\theta}(c_i),\quad\text{and}\quad
\Hat{\Upsilon}(\bm{\hat{\theta}})
=\frac{1}{n}\sum_{i=1}^{n}{\bf{1}}_{\left[d^2(\bm{\theta_i})\leq d^2({\bm{\Hat{\theta}}})\right]}.
\end{align}
This construction ensures that conjectures which are both semantically related and structurally similar in $\bm{\theta}$-space receive the highest weights, so that the estimated motif and non-triviality score are inherited from the most relevant element in $R$. 

\medskip
Note that the proposed scores do not claim to determine provability or unprovability. Instead, it quantifies structural proximity to conjectures that are empirically resistant to existing proof techniques and gives us an idea of the complexity of the problem. 

\begin{remark}
The Dirac delta formulation of $\mathcal{H}$ preserves the exact geometry of the exemplar set without introducing distributional assumptions unsupported by the data, as any parametric smoothing would artificially interpolate between conjectures occupying meaningfully distinct positions in ${\bm{\theta}}$-space. The parameter $\lambda=0.7$ assigns dominant weight to the structural motif $\Hat{\theta}(c_{\mathrm{new}})$ over linguistic content in the hybrid representation, reflecting the principle that mathematical depth and complexity, as encoded by $(\theta_1,\ldots,\theta_6)$, are more discriminative of conjecture similarity than surface level semantic content. The temperature parameter $\tau=0.01$ concentrates the softmax weights sharply on the nearest reference conjectures in hybrid representation space, ensuring that $\hat{\theta}(c_{\mathrm{new}})$ and $\Upsilon(c_{\mathrm{new}})$ are inherited from the most structurally relevant element in $R$ rather than diluted by distant members.    
\end{remark}


\section{Experimental Results}
\noindent
The non-triviality score $\Upsilon(c)\in[0,1]$ measures the structural position of a known conjecture within the empirical hardness distribution $\mathcal{H}$, calibrated against 18 reference conjectures from analytic number theory with the Riemann Hypothesis as the ceiling at $\Upsilon=1$. A score near 1 indicates that the conjecture is a structural outlier, simultaneously extreme across all six theta dimensions, as RH is. A score near zero does not indicate mathematical triviality. Rather, it indicates that the conjecture occupies the structural centre of the known hardness landscape, which is itself a region of profound mathematical depth. For example, in our current conjecture space ${\bm{\theta}}$ (as given in Table~\ref{conjecturespace}), Chowla's conjecture on prime correlations, widely regarded as harder and deeper than RH, scores $\Upsilon=0$ precisely because its structural profile is balanced and central rather than extreme. 

\subsection{Example}\label{benchmarkexample}
In this section, we present and interpret the benchmark scores for Conjecture~\ref{a1}. Before examining its scoring, we present the empirical distribution of the known conjectures with respect to their assigned score vectors, as given in Table~\ref{knownconjectures}.

\begin{table}[ht]
\centering
\begin{tabular}{|c|c|c|c|}
\hline
Rank & Conjecture & $d^2$ & $\Upsilon$ \\
\hline
1  & Chowla's Conjecture & 1.3671 & 0.0000 \\
\hline
2  & Twin Prime Conjecture & 2.8386 & 0.0588 \\
\hline
3  & Polignac's Conjecture & 3.3000 & 0.1176 \\
\hline
4  & Generalised Elliott-Halberstam & 4.4167 & 0.1765 \\
\hline
5  & Elliott-Halberstam Conjecture & 4.5940 & 0.2353 \\
\hline
6  & Granville's Refinement of Cram\'{e}r's Conjecture & 5.0593 & 0.3529 \\
\hline
7  & Hardy-Littlewood Second Conjecture & 5.0593 & 0.3529 \\
\hline
8  & Shanks Conjecture & 5.5015 & 0.4118 \\
\hline
9  & Parity Problem & 5.5948 & 0.4706 \\
\hline
10 & Bateman-Horn Conjecture & 5.7381 & 0.5294 \\
\hline
11 & Landau's Fourth Problem (Near-Square Primes) & 6.1323 & 0.5882 \\
\hline
12 & Dickson's Conjecture & 6.3513 & 0.6471 \\
\hline
13 & Schinzel's Hypothesis H & 6.8698 & 0.7059 \\
\hline
14 & Firoozbakht's Conjecture & 7.0127 & 0.7647 \\
\hline
15 & Hardy-Littlewood Prime k-tuples Conjecture & 7.1207 & 0.8235 \\
\hline
16 & Bunyakovsky Conjecture & 7.6289 & 0.8824 \\
\hline
17 & Legendre's Conjecture & 7.9739 & 0.9412 \\
\hline
18 & Riemann Hypothesis & 9.4408 & 1.0000 \\
\hline
\end{tabular}
\caption{Conjecture space of known conjectures}
\label{conjecturespace}
\end{table}

\noindent
Conjecture~\ref{a1} is stated entirely in terms of the function $B_Q(x)$ as defined in~\eqref{bqdefinition}, a definition introduced for the first time by HypothesiX. Its language, notation, and structural formulation share no surface-level overlap with either the Twin Prime Conjecture or the Elliott-Halberstam Conjecture. Nevertheless, the benchmark assigns
\begin{align*}
\hat{\bm{\theta}} = [8.01, 8.16, 7.0, 5.1, 6.07, 5.98],    
\end{align*}
\noindent
a Mahalanobis distance $d^2 = 3.70$, and a hybrid similarity of $0.932$ to the Elliott-Halberstam conjecture, placing it geometrically between the twin prime conjecture and the Elliott-Halberstam conjecture
\begin{align}\label{d2forbq}
d^2_{\text{twin prime}} \approx 2.84 
\leq d^2_{\text{Conjecture A.1}} \approx 3.70 
\leq d^2_{\text{Elliott-Halberstam}} \approx 4.42,
\end{align}
in the conjecture embedding. This reflects the benchmark's ability to recover the genuine mathematical content of the conjecture, namely that it encodes information about the distribution of twin primes in arithmetic progressions, which is precisely the intersection of those two conjectures.

\medskip
\noindent
Crucially, we provide an elementary proof of Theorem~\ref{pi2bq} 
(a weaker version of Conjecture~\ref{a1}), yet the benchmark scores do not reflect this. This is by design. The benchmark is not measuring provability; it is measuring the third condition of the Birch test, which asks whether a conjecture carries sufficient information to spark new research. The most directly relevant component of $\hat{\bm{\theta}}$ for this condition is $\hat{\theta}_5 \approx 6.07$, which estimates the number of results that could meaningfully build upon the conjecture across mathematical domains. A score of $6.07$ out of $10$ on this dimension, for an entirely new and provable statement, indicates that the information it introduces has genuine structural utility for further research, particularly in sieve-theoretic approaches to twin prime distribution. Additionally, $\hat{\theta}_1 \approx 8.01$ suggests a rich equivalence structure, meaning the conjecture's content is likely expressible in multiple useful 
forms, further supporting its research potential.

\medskip
\noindent
Taken together, $\hat{\theta}_2 \approx 8.16$ reflects the axiomatic depth of the territory the conjecture operates in, $\hat{\theta}_3 \approx 7.0$ indicates strong connections to the existing literature, and $\hat{\theta}_6 \approx 5.98$ suggests meaningful reductions to known hard problems. The comparatively lower $\hat{\theta}_4 \approx 5.1$ is consistent with the fact that $B_Q(x)$ involves residue-class minimisation, making direct computational verification across all moduli $Q$ non-trivial. Together, these scores characterise a conjecture that is elementary to establish yet mathematically non-trivial in the sense that matters for the Birch test: it introduces a new structural object whose properties are connected to deep open problems and from which further research can be developed, as consistent with our argument in Section~\ref{selectedinequalities}. Hence, this benchmark captures the essence of the non-triviality condition of the Birch test in a quantifiable manner.

\subsubsection{Significance of the Mahalanobis Distance}\label{d2importance}
The Mahalanobis distance $d^2(\bm{{\theta}})$ was introduced in Section~\ref{mahalanobis} as a measure of structural proximity to the empirical hardness distribution of known conjectures.  We now explain that this distance has a deeper significance. It can be used to localise the errors in machine-generated conjectures at scale. 

\medskip
We explain the above claim with an example. Consider the statement of Conjecture~\ref{a1} states that for all squarefree $Q$ such that $6 \mid Q$,
\begin{align}\label{incorrectrelation}
\pi_2(x)\leq B_Q(x)+2
\end{align}
for all $x\geq 7$. Let us assume that this is not universally true. We do not believe Conjecture~\ref{a1} to be false. In fact, given the Hardy-Littewood $k$-tuples conjecture, it is likely to hold for all $Q$ and $x$ satisfying the stated conditions\footnote{There are, however, incorrect conjectures generated by HypothesiX, as can be seen in the full list of conjectures. We nevertheless analyse this conjecture under the assumption that it is not universally true, purely to avoid introducing additional number-theoretic jargon into the article. Analysing a different problem would require explaining its background and related research, thereby moving away from the main focus of this article.}. If it turns out to be false for any $Q$, it might be due to the additive constant 2 being insufficient to absorb the contribution from small primes for larger value of $Q$.

\medskip
\noindent
Observe that from~\eqref{d2forbq}, we know that the statement lies between the Twin Prime Conjecture and the Elliott–Halberstam Conjecture in the $\bm{\theta}$-space. It is evident that the left-hand side of the inequality has the twin prime counting function; therefore, Conjecture~\ref{a1} is close to the Twin Prime Conjecture in the conjecture space. Recalling the Elliott–Halberstam conjecture, it states that the error function
\begin{align*}
E(x;Q)
\coloneqq \max_{(a,Q)=1} \left|\pi(x;Q,a)-\frac{\pi(x)}{\varphi(Q)}\right|,
\end{align*}
is bounded by
\begin{align*}
\sum_{1\leq Q\leq x^{\theta}}|E(x;Q)|\ll \frac{x}{\log^A x},
\end{align*}
for every $\theta<1$ and $x>2$, with $A>0$. Noting the definition of $B_Q(x)$, as given in~\eqref{bqdefinition} and the statement of the Elliot-Halberstam conjecture the benchmark indicate us to understand the connection between them with the twin prime conjecture which eventually help us fixing the error. Before writing the argument on how we fix the error in~\eqref{incorrectrelation}, it is worth noting that for a number theorist, the error and the correct statement is obvious. However, for machine generating mathematics at scale, the proposed distance $d^2(\bm{\theta})$ works as an error indication signal, making it easier for a non-expert or a machine to look at the right neighbourhood within the vast body of knowledge.

\medskip
\noindent
From the work of Maynard~\cite{James2} and Tao, we know that the lower bound for the gaps between primes $p_{n+1}-p_n\leq 12$, under the assumption of Elliot-Halberstam conjecture, with $p_n$ denoting the $n$-th prime. If the right hand side become 2, then it will solve the twin prime conjecture. This is the most direct known connection between these two problems within the known literature. While the Maynard-Tao sieve gives a lower bound, the conjecture in hand propose an upper bound with the twin prime counting function. Thus, looking at the error term of the Elliot-Halberstam conjecture, we note that $\pi(x;Q,a)\sim \pi(x)/\varphi(Q)$. However, roughly speaking, to understand the incorrectness in~\eqref{incorrectrelation}, we check the asymptotic estimate for the left and right hand side and see the relationship between them.
\begin{align*}
\mathrm{LHS:}& \;\pi_2(x)\sim 2C_2\frac{x}{(\log x)^2}\quad (\text{by Conjecture~\ref{hardylittewood}})\\
\mathrm{RHS:}& \; B_Q(x)\approx\pi(x)/\varphi(Q)\sim \frac{x}{\varphi(Q)\log x}.
\end{align*}
\noindent
From the aforementioned expression, it becomes evident that the objects such as $\pi_2(x)$ and $B_2(x)$, as well as the direction of the inequality, are elementarily correct, and therefore the incorrectness lies in the constant.

\begin{remark}
If a conjecture lies far from all known conjectures in the $\bm{\theta}$-space, it indicates one of two possibilities. Either it is a genuinely important discovery, or it is a malformed result. As mentioned earlier, this benchmark cannot formally verify the correctness of a conjecture. Rather, the Mahalanobis distance $d^2(\bm{\theta})$ gives a structural comparison against known results, indicating where in mathematical knowledge a machine generated new conjecture sits. In this sense, the Mahalanobis distance extends the Birch test benchmark from a scoring system into a first-pass diagnostic that can reduce the manual evaluation that frameworks such as~\cite{alex} relied upon entirely.
\end{remark}

\subsection{Limitations}
In this section, we discuss about the key limitations of the proposed benchmark. 
\begin{enumerate}
\item We have relied on the literature review to assign scores for the vector ${\bm{\theta}(c_i)}_{i=1}^{18}$, making the framework sensitive to the consistency of the reference table. Since $n=18$ in the examplar set the scoring resolution is limited to intervals of $\frac{1}{17}\approx 0.059$, and a larger reference set would yield an accurate hardness scoring.  

\item The estimated structural value of $\bm{\hat{\theta}}$ for new conjectures is a convex combination of ${\bm{\theta}}$ vector, constraining it to the convex hull of the examplar set. Conjectures whose true structural motif lies outside this envelope will consequently be underscored. 
\end{enumerate}

\subsection{Results for non-triviality Scoring}
\noindent
We present two key results in this section. First, we provide the values of $\bm{\hat{\theta}}$ and $\Upsilon'$ as defined in~\eqref{thetahatdefinition} for all 78 conjectures listed in AppendiX~\ref{appendix}, and available on GitHub\footnote{\url{https://github.com/MadhuparnaDas96/HypothesiX-Benchmark}} along with supported codes for correctness verification. Furthermore, we construct a clustering model to compare the distribution of the generated results with known unsolved problems and established theorems in the literature.


\subsubsection{Hardness Scoring}

In this section, we present the values of $\bm{\hat{\theta}}$ and $\Hat{\Upsilon}$ in Table~\ref{benchmarktable1} and Table~\ref{benchmarktable2} for the generated inequalities and the conjectures related to the distribution of prime $k$-tuples listed in the GitHub repository. 

\begin{table}[ht]
\begin{center}
\begin{tabular}{|c|c|c|c|c|}
\hline
$c_{new}$ Name & ${\bm{\Hat{\theta}}}=[{\Hat{\theta}}_1,{\Hat{\theta}}_2,{\Hat{\theta}}_3,{\Hat{\theta}}_4,{\Hat{\theta}}_5,{\Hat{\theta}}_6]$ & $d^2(\bm{\Hat{\theta}})$ & $\Hat{\Upsilon}({\bm{\Hat{\theta}}})$ & Closest known conjecture \\
\hline
A.1  & $[8.01,\ 8.16,\ 7.00,\ 5.10,\ 6.07,\ 5.98]$ & 3.7061 & 0.1667 & Elliott--Halberstam\\
A.2  & $[7.93,\ 8.31,\ 7.05,\ 5.20,\ 6.09,\ 5.88]$ & 2.8243 & 0.0556 & Elliott--Halberstam \\
A.3  & $[7.69,\ 8.75,\ 6.97,\ 5.93,\ 6.08,\ 5.25]$ & 0.6794 & 0.0000 & Elliott--Halberstam \\
A.4  & $[8.00,\ 8.20,\ 7.01,\ 5.06,\ 6.10,\ 6.03]$ & 3.6277 & 0.1667 & Elliott--Halberstam \\
A.5  & $[8.00,\ 8.19,\ 7.01,\ 5.08,\ 6.09,\ 6.02]$ & 3.6224 & 0.1667 & Elliott--Halberstam \\
A.6  & $[8.11,\ 8.52,\ 6.90,\ 5.34,\ 6.22,\ 5.95]$ & 2.2906 & 0.0556 & Elliott--Halberstam \\
A.7  & $[7.90,\ 8.85,\ 6.89,\ 4.93,\ 6.31,\ 6.12]$ & 2.1612 & 0.0556 & Generalized Elliott--Halberstam \\
A.8  & $[6.66,\ 7.06,\ 6.16,\ 8.54,\ 5.11,\ 4.29]$ & 0.7121 & 0.0000 & Shanks Conjecture\\
A.9  & $[8.72,\ 9.52,\ 8.53,\ 8.61,\ 6.72,\ 4.24]$ & 4.1739 & 0.1667 & Hardy--Littlewood Prime $k$-tuples \\
A.10 & $[8.03,\ 8.25,\ 7.07,\ 5.24,\ 6.12,\ 5.95]$ & 3.1617 & 0.1111 & Elliott--Halberstam \\
A.11 & $[4.03,\ 8.01,\ 5.02,\ 3.03,\ 3.03,\ 2.02]$ & 4.9842 & 0.2778 & HL Second Conjecture \\
A.12 & $[4.11,\ 8.02,\ 5.07,\ 3.08,\ 3.10,\ 2.11]$ & 4.7892 & 0.2778 & HL Second Conjecture \\
A.13 & $[5.38,\ 8.25,\ 5.80,\ 3.84,\ 4.18,\ 3.41]$ & 2.4850 & 0.0556 & HL Second Conjecture \\
A.14 & $[8.03,\ 8.32,\ 7.11,\ 5.24,\ 6.18,\ 5.97]$ & 2.9402 & 0.1111 & Elliott--Halberstam \\
A.15 & $[7.99,\ 8.48,\ 7.15,\ 5.26,\ 6.29,\ 6.04]$ & 2.4232 & 0.0556 & Elliott--Halberstam  \\
A.16 & $[8.08,\ 8.98,\ 7.58,\ 6.55,\ 6.49,\ 5.50]$ & 0.6322 & 0.0000 & Chowla\\
A.17 & $[4.11,\ 8.04,\ 5.09,\ 3.09,\ 3.11,\ 2.13]$ & 4.7940 & 0.2778 & HL Second Conjecture \\
A.18 & $[4.90,\ 5.70,\ 5.02,\ 8.24,\ 4.52,\ 3.50]$ & 2.6120 & 0.0556 & Shanks Conjecture\\
A.19 & $[8.00,\ 8.72,\ 6.48,\ 4.53,\ 6.09,\ 6.00]$ & 3.2349 & 0.1111 & Generalized EH  \\
A.20 & $[7.99,\ 8.23,\ 6.96,\ 4.99,\ 6.08,\ 6.03]$ & 3.6467 & 0.1667 & Elliott--Halberstam \\
A.21 & $[7.98,\ 8.40,\ 7.03,\ 5.08,\ 6.20,\ 6.09]$ & 2.9499 & 0.1111 & Elliott--Halberstam  \\
A.22 & $[8.01,\ 8.48,\ 7.06,\ 5.20,\ 6.26,\ 6.09]$ & 2.5644 & 0.0556 & Elliott--Halberstam  \\
A.23 & $[8.00,\ 8.43,\ 6.83,\ 4.88,\ 6.11,\ 6.01]$ & 3.1494 & 0.1111 & Elliott--Halberstam \\
A.24 & $[8.03,\ 8.49,\ 6.71,\ 4.80,\ 6.09,\ 5.99]$ & 3.1858 & 0.1111 & Elliott--Halberstam  \\
A.25 & $[7.99,\ 8.38,\ 6.94,\ 5.03,\ 6.14,\ 6.04]$ & 3.0473 & 0.1111 & Elliott--Halberstam \\
A.26 & $[8.03,\ 8.48,\ 7.08,\ 5.22,\ 6.28,\ 6.14]$ & 2.5927 & 0.0556 & Elliott--Halberstam  \\
A.27 & $[6.71,\ 7.37,\ 6.68,\ 8.09,\ 5.59,\ 4.39]$ & 0.2552 & 0.0000 & Shanks Conjecture\\
A.28 & $[8.04,\ 8.33,\ 7.04,\ 5.14,\ 6.18,\ 6.01]$ & 3.0968 & 0.1111 & Elliott--Halberstam \\
A.29 & $[4.21,\ 5.17,\ 4.22,\ 8.08,\ 4.11,\ 3.10]$ & 4.7127 & 0.2778 & Shanks Conjecture \\
A.30 & $[8.02,\ 8.16,\ 6.95,\ 5.02,\ 6.05,\ 5.99]$ & 3.8598 & 0.1667 & Elliott--Halberstam \\
A.31 & $[8.00,\ 8.16,\ 6.97,\ 5.02,\ 6.05,\ 5.99]$ & 3.8427 & 0.1667 & Elliott--Halberstam \\
A.32 & $[8.04,\ 8.32,\ 6.98,\ 5.10,\ 6.14,\ 5.98]$ & 3.1710 & 0.1111 & Elliott--Halberstam \\
A.33 & $[8.03,\ 8.52,\ 7.09,\ 5.27,\ 6.28,\ 6.00]$ & 2.3294 & 0.0556 & Elliott--Halberstam \\
A.34 & $[8.03,\ 8.25,\ 6.98,\ 5.08,\ 6.10,\ 5.98]$ & 3.4188 & 0.1667 & Elliott--Halberstam\\
A.35 & $[7.99,\ 8.44,\ 6.88,\ 4.92,\ 6.14,\ 6.00]$ & 3.0290 & 0.1111 & Elliott--Halberstam \\
A.36 & $[8.01,\ 8.18,\ 7.04,\ 5.15,\ 6.09,\ 5.98]$ & 3.5453 & 0.1667 & Elliott--Halberstam \\
A.37 & $[8.00,\ 8.16,\ 6.95,\ 4.98,\ 6.05,\ 6.00]$ & 3.9318 & 0.1667 & Elliott--Halberstam\\
A.38 & $[8.01,\ 8.29,\ 6.82,\ 4.86,\ 6.05,\ 6.00]$ & 3.6299 & 0.1667 & Elliott--Halberstam \\
A.39 & $[8.05,\ 8.25,\ 7.03,\ 5.13,\ 6.16,\ 6.06]$ & 3.4141 & 0.1667 & Elliott--Halberstam  \\
A.40 & $[7.91,\ 8.73,\ 7.27,\ 5.87,\ 6.34,\ 5.73]$ & 1.0052 & 0.0000 & Elliott--Halberstam \\
A.41 & $[8.04,\ 8.29,\ 6.97,\ 5.08,\ 6.12,\ 6.01]$ & 3.3141 & 0.1667 & Elliott--Halberstam \\
A.42 & $[8.07,\ 8.35,\ 7.03,\ 5.20,\ 6.18,\ 6.00]$ & 2.9597 & 0.1111 & Elliott--Halberstam  \\
A.43 & $[8.49,\ 9.01,\ 7.99,\ 7.67,\ 6.58,\ 4.87]$ & 1.4792 & 0.0556 & HL Prime $k$-tuples \\
A.44 & $[8.43,\ 8.99,\ 7.88,\ 7.65,\ 6.57,\ 4.97]$ & 1.1987 & 0.0000 & HL Prime $k$-tuples \\
\hline
\end{tabular}
\end{center}
\caption{Hardness Scoring for the generated inequalities}
\label{benchmarktable1}
\end{table}

\begin{table}[ht]
\begin{center}
\begin{tabular}{|c|c|c|c|c|}
\hline
$c_{new}$ Name & ${\bm{\Hat{\theta}}}=[{\Hat{\theta}}_1,{\Hat{\theta}}_2,{\Hat{\theta}}_3,{\Hat{\theta}}_4,{\Hat{\theta}}_5,{\Hat{\theta}}_6]$ & $d^2(\bm{\Hat{\theta}})$ & $\Hat{\Upsilon}({\bm{\Hat{\theta}}})$ & Closest known conjecture \\
\hline
A.45  & $[8.04, 8.24, 6.97, 5.07, 6.09, 5.98]$ & $3.5046$ & $0.1667$ & Elliott-Halberstam  \\
A.46 & $[8.03, 8.42, 6.87, 4.97, 6.14, 6.00]$ & $3.0618$ & $0.1111$ & Elliott-Halberstam \\
A.47  & $[8.22, 8.92, 7.64, 6.21, 6.80, 5.98]$ & $1.0370$ & $0.0000$ & Chowla\\
A.48  & $[8.81, 8.95, 7.48, 7.58, 6.51, 5.34]$ & $1.2594$ & $0.0000$ & Bateman-Horn \\
A.49 & $[7.72, 9.14, 7.31, 5.11, 6.56, 6.31]$ & $1.8176$ & $0.0556$ & Parity Problem \\
A.50  & $[7.64, 8.68, 6.89, 5.04, 6.07, 5.65]$ & $1.6574$ & $0.0556$ & Elliott-Halberstam  \\
A.51  & $[8.05, 8.36, 7.01, 5.21, 6.16, 5.95]$ & $2.8451$ & $0.1111$ & Elliott-Halberstam \\
A.52  & $[8.68, 8.92, 7.50, 7.72, 6.45, 5.22]$ & $1.1114$ & $0.0000$ & Bateman-Horn  \\
A.53 & $[8.09, 8.52, 6.99, 5.26, 6.26, 5.99]$ & $2.4081$ & $0.0556$ & Elliott-Halberstam \\
A.54 & $[8.22, 9.05, 7.90, 6.56, 6.93, 5.97]$ & $0.8728$ & $0.0000$ & Chowla \\
A.55  & $[8.30, 8.79, 7.26, 6.16, 6.43, 5.66]$ & $1.1495$ & $0.0000$ & Elliott-Halberstam \\
A.56  & $[4.30, 8.10, 5.21, 3.24, 3.28, 2.29]$ & $4.3688$ & $0.1667$ & Granville's Refined Cram\'{e}r \\
A.57  & $[8.04, 8.55, 7.14, 5.36, 6.30, 6.04]$ & $2.1861$ & $0.0556$ & Elliott-Halberstam \\
A.58  & $[7.98, 8.54, 6.81, 4.84, 6.15, 6.02]$ & $2.8952$ & $0.1111$ & Elliott-Halberstam \\
A.59 & $[8.01, 8.38, 6.70, 4.73, 6.03, 5.99]$ & $3.5619$ & $0.1667$ & Elliott-Halberstam \\
A.60  & $[8.02, 8.41, 6.90, 4.98, 6.14, 6.01]$ & $3.0428$ & $0.1111$ & Elliott-Halberstam  \\
A.61 & $[8.02, 8.45, 6.88, 4.97, 6.15, 6.02]$ & $2.9794$ & $0.1111$ & Elliott-Halberstam  \\
A.62  & $[8.00, 8.65, 6.81, 4.87, 6.22, 6.01]$ & $2.6322$ & $0.0556$ & Generalized EH \\
A.63  & $[8.02, 8.27, 6.97, 5.06, 6.11, 6.01]$ & $3.3794$ & $0.1667$ & Elliott-Halberstam  \\
A.64  & $[8.02, 8.27, 6.99, 5.11, 6.11, 6.00]$ & $3.2841$ & $0.1111$ & Elliott-Halberstam \\
A.65  & $[8.02, 8.33, 6.99, 5.20, 6.13, 5.95]$ & $2.9041$ & $0.1111$ & Elliott-Halberstam \\
A.66  & $[8.04, 8.36, 6.96, 5.09, 6.13, 5.97]$ & $3.0539$ & $0.1111$ & Elliott-Halberstam  \\
A.67  & $[6.56, 7.23, 6.76, 8.27, 5.53, 4.33]$ & $0.4060$ & $0.0000$ & Shanks Conjecture \\
A.68 & $[7.86, 8.84, 7.30, 5.68, 6.38, 5.90]$ & $1.1337$ & $0.0000$ & Elliott-Halberstam \\
A.69  & $[8.07, 8.78, 7.24, 5.90, 6.34, 5.74]$ & $1.1094$ & $0.0000$ & Elliott-Halberstam  \\
A.70  & $[8.00, 8.52, 6.77, 4.81, 6.14, 6.00]$ & $3.0075$ & $0.1111$ & Elliott-Halberstam  \\
A.71 & $[8.00, 8.82, 6.30, 4.32, 6.06, 6.00]$ & $3.6216$ & $0.1667$ & Generalized EH \\
A.72  & $[8.00, 8.59, 6.48, 4.48, 6.03, 6.00]$ & $3.5456$ & $0.1667$ & Generalized EH \\
A.73  & $[7.97, 8.52, 7.04, 5.10, 6.25, 6.04]$ & $2.5210$ & $0.0556$ & Elliott-Halberstam \\
A.74  & $[8.00, 8.45, 6.98, 5.06, 6.18, 6.01]$ & $2.7829$ & $0.0556$ & Elliott-Halberstam  \\
A.75  & $[7.98, 8.79, 6.72, 4.76, 6.23, 6.00]$ & $2.5863$ & $0.0556$ & Generalized EH \\
A.76  & $[8.01, 8.39, 6.76, 4.79, 6.07, 5.99]$ & $3.4203$ & $0.1667$ & Elliott-Halberstam \\
A.77  & $[8.01, 8.70, 6.49, 4.51, 6.09, 6.00]$ & $3.2894$ & $0.1111$ & Generalized EH \\
A.78  & $[8.01, 8.52, 6.82, 4.86, 6.16, 5.99]$ & $2.9455$ & $0.1111$ & Elliott-Halberstam  \\
\hline
\end{tabular}
\end{center}
\caption{Hardness Scoring for the generated conjectures related to the prime $k$-tuples}
\label{benchmarktable2}
\end{table}

We also visualise the conjecture cluster as a two-dimensional map. Each point is a conjecture or result, projected onto the first two principal components of the six-dimensional feature space $\bm{\theta}$, learned from the 18 reference conjectures. Deep blue circles are known reference conjectures, and the light blue circle marks the Riemann Hypothesis as the ceiling. The plus sign $(+)$ is the centroid of the reference cluster. New conjectures appear as orange to red circles, where the colour encodes the non-triviality score $\Hat{\Upsilon}$: dark orange indicates $\Hat{\Upsilon}\to 0$ and light orange $\Hat{\Upsilon}\to 1$. Green filled circles represent proven theorems from existing literature, positioned via softmax weighted interpolation of reference theta vectors using text embedding similarity\footnote{All points are clipped to a bounding circle centred on the data range midpoint.}.

\begin{figure}[H]
\centering
\begin{subfigure}[b]{0.45\linewidth}
\centering
\includegraphics[width=\linewidth]{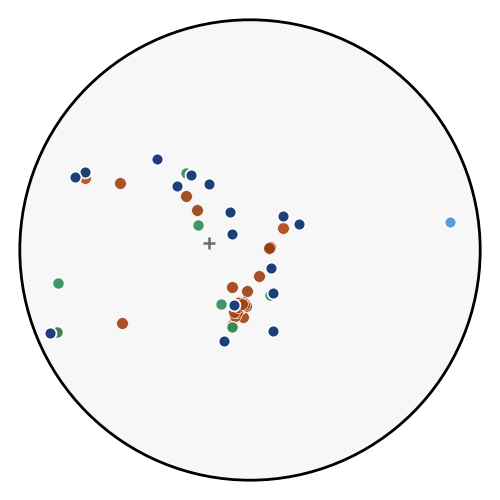}
\caption{$\bm{\theta}$ space for Conjectures A.1 to A.44}
\label{fig:sub1}
\end{subfigure}
\hfill
\begin{subfigure}[b]{0.45\linewidth}
\centering
\includegraphics[width=\linewidth]{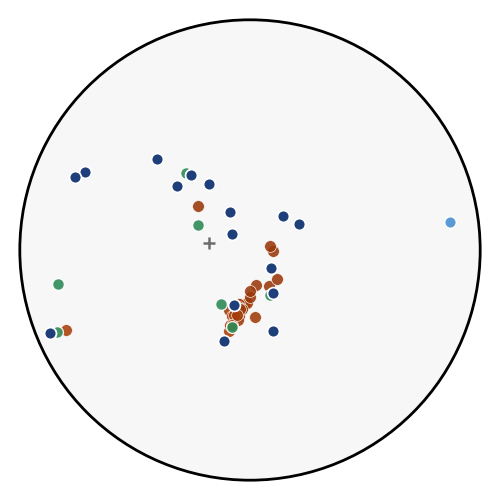}
\caption{$\bm{\theta}$ space for Conjectures A.45 to A.78}
\label{fig:sub2}
\end{subfigure}
\caption{Projection of the conjecture cluster in the learned $\theta$-space obtained from the six-dimensional feature vectors $\theta = (\theta_1,\ldots,\theta_6)$ defined in Section~\ref{nontriviality}. Each point represents either a known conjecture from the reference set $R$, a machine-generated conjecture produced by HypothesiX, or a related theorem from the literature, after projection onto the first two principal components of the feature space. Known conjectures are shown in blue, with lighter blue corresponding to higher non-triviality score $\Upsilon$ and darker blue to lower $\Upsilon$; the light blue point denotes the Riemann Hypothesis, which serves as the ceiling element of the hardness distribution. Generated conjectures are shown in orange-to-red, where lighter shades indicate larger estimated non-triviality scores $\hat{\Upsilon}$. Green points correspond to established theorems from the literature, embedded into the same space using softmax-weighted interpolation of the reference $\theta$-vectors via text-embedding similarity. The plus sign marks the centroid of the reference conjecture cluster. All points are clipped to a bounding circle centred at the midpoint of the projected data range.}
\label{fig:main}
\end{figure}

\subsubsection{$\bm{\hat{\theta}}$-space}

We have created cluster models for only for generated conjectures. For each conjecture $c_{\mathrm{new}}$, we compute the Mahalanobis distance with empirical mean $\hat{\mu}$ and covariance $\hat{\Sigma}$ as in Section~\ref{mahalanobis}. The distance $\hat{d}^2(\bm{\hat{\theta}})$ and the non-triviality scoring $\Upsilon(\bm{\hat{\theta}}^{(i)})$ as defined in~\eqref{d2formula} and~\eqref{upsilon} respectively, are given in Table~\ref{newthetaspacetable} and Figure~\ref{newcluster}.

\begin{table}[ht]
\centering
\begin{tabular}{|c|c|c||c|c|c|}
\hline
$c_{\mathrm{new}}$ Name & $\hat{d}^2$ & $\Upsilon(\bm{\hat{\theta}}^{(i)})$ &$c_{\mathrm{new}}$ Name & $\hat{d}^2$ & $\Upsilon(\bm{\hat{\theta}}^{(i)})$ \\
\hline
A.1  & 2.5662  & 0.649 & A.40 & 2.4650  & 0.584 \\
A.2  & 1.2294  & 0.338 & A.41 & 1.0832  & 0.273 \\
A.3  & 2.4949  & 0.597 & A.42 & 0.7263  & 0.156 \\
A.4  & 2.1097  & 0.545 & A.43 & 13.4898 & 0.844 \\
A.5  & 2.2239  & 0.571 & A.44 & 9.4676  & 0.805 \\
A.6  & 0.8562  & 0.182 & A.45 & 1.6763  & 0.481 \\
A.7  & 5.6151  & 0.714 & A.46 & 0.5704  & 0.091 \\
A.8  & 45.2689 & 1.000 & A.47 & 8.6734  & 0.792 \\
A.9  & 31.5297 & 0.974 & A.48 & 15.9945 & 0.896 \\
A.10 & 1.8285  & 0.519 & A.49 & 24.5856 & 0.961 \\
A.11 & 17.0916 & 0.935 & A.50 & 1.7105  & 0.494 \\
A.12 & 16.3041 & 0.909 & A.51 & 0.6270  & 0.117 \\
A.13 & 7.2164  & 0.766 & A.52 & 15.8260 & 0.883 \\
A.14 & 1.4082  & 0.403 & A.53 & 0.3581  & 0.013 \\
A.15 & 1.2209  & 0.312 & A.54 & 13.7411 & 0.857 \\
A.16 & 4.6216  & 0.688 & A.55 & 2.2114  & 0.558 \\
A.17 & 16.3346 & 0.922 & A.56 & 14.7460 & 0.870 \\
A.18 & 24.3066 & 0.948 & A.57 & 0.9520  & 0.234 \\
A.19 & 7.1468  & 0.753 & A.58 & 0.8813  & 0.208 \\
A.20 & 1.6630  & 0.468 & A.59 & 1.6432  & 0.455 \\
A.21 & 0.9705  & 0.247 & A.60 & 0.4552  & 0.039 \\
A.22 & 0.8668  & 0.195 & A.61 & 0.4392  & 0.026 \\
A.23 & 0.5997  & 0.104 & A.62 & 1.9824  & 0.532 \\
A.24 & 1.7258  & 0.506 & A.63 & 1.2214  & 0.325 \\
A.25 & 0.5686  & 0.078 & A.64 & 1.2758  & 0.364 \\
A.26 & 1.3060  & 0.390 & A.65 & 0.6837  & 0.143 \\
A.27 & 7.8625  & 0.779 & A.66 & 0.6559  & 0.130 \\
A.28 & 1.0057  & 0.260 & A.67 & 12.3073 & 0.818 \\
A.29 & 43.9392 & 0.987 & A.68 & 6.8352  & 0.740 \\
A.30 & 2.5153  & 0.636 & A.69 & 1.6283  & 0.442 \\
A.31 & 2.5715  & 0.662 & A.70 & 1.2996  & 0.377 \\
A.32 & 0.9349  & 0.221 & A.71 & 13.4352 & 0.831 \\
A.33 & 0.4988  & 0.065 & A.72 & 5.4325  & 0.701 \\
A.34 & 1.5340  & 0.429 & A.73 & 0.7391  & 0.169 \\
A.35 & 0.4773  & 0.052 & A.74 & 0.3411  & 0.000 \\
A.36 & 2.5097  & 0.623 & A.75 & 4.5622  & 0.675 \\
A.37 & 2.4967  & 0.610 & A.76 & 1.2141  & 0.299 \\
A.38 & 1.2306  & 0.351 & A.77 & 6.7691  & 0.727 \\
A.39 & 1.5115  & 0.416 & A.78 & 1.1004  & 0.286 \\
\hline
\end{tabular}
\caption{$\bm{\hat{\theta}}$-space for generated  Conjectures}
\label{newthetaspacetable}
\end{table}

\begin{figure}[H]
\centering
\begin{subfigure}[b]{0.3\linewidth}
\centering
\includegraphics[width=\linewidth]{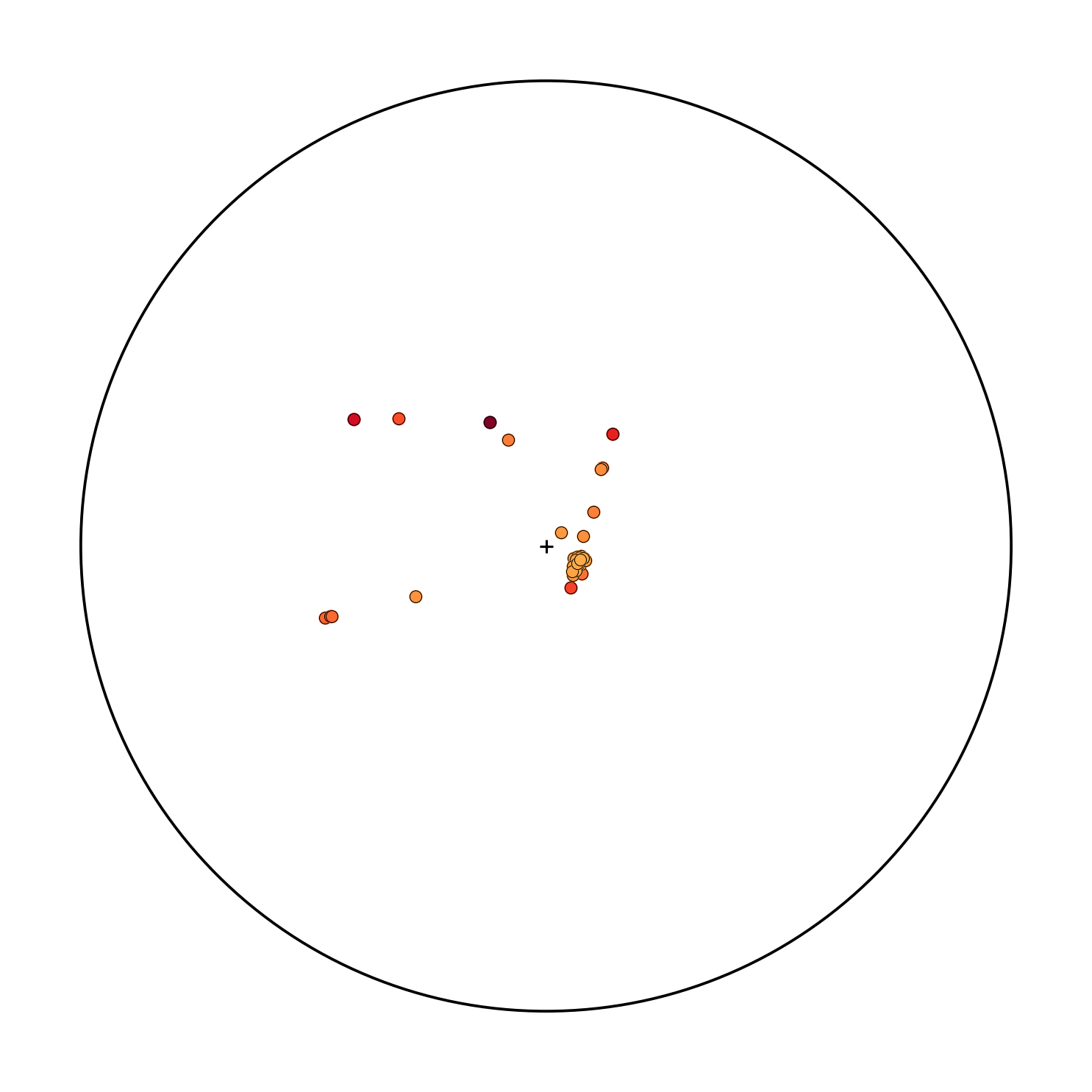}
\caption{$\bm{\hat{\theta}}$ space for Conjectures A.1 to A.44}
\label{fig:new1}
\end{subfigure}
\hfill
\begin{subfigure}[b]{0.3\linewidth}
\centering
\includegraphics[width=\linewidth]{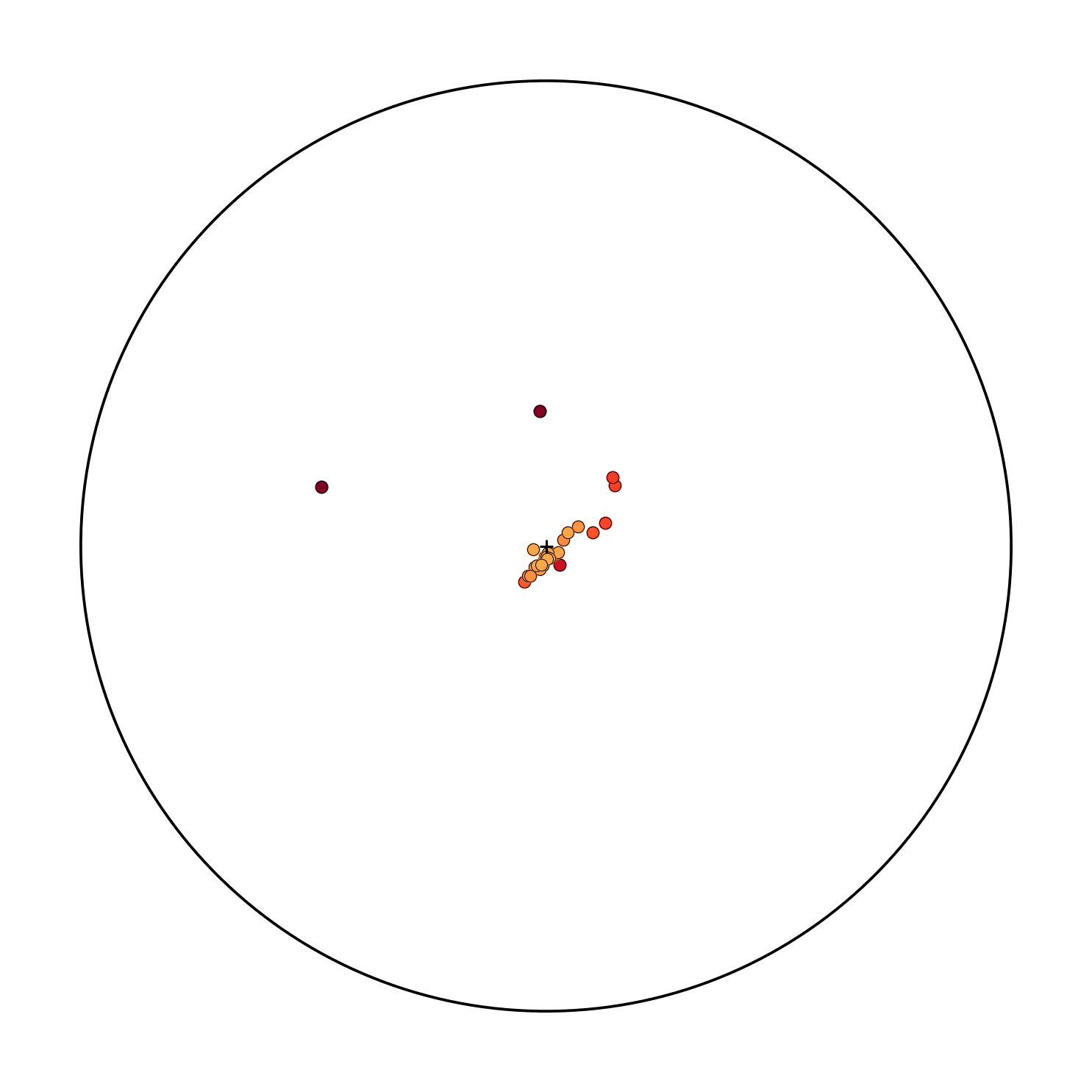}
\caption{$\bm{\hat{\theta}}$ space for Conjectures A.45 to A.78}
\label{fig:new2}
\end{subfigure}
\hfill
\begin{subfigure}[b]{0.3\linewidth}
\centering
\includegraphics[width=\linewidth]{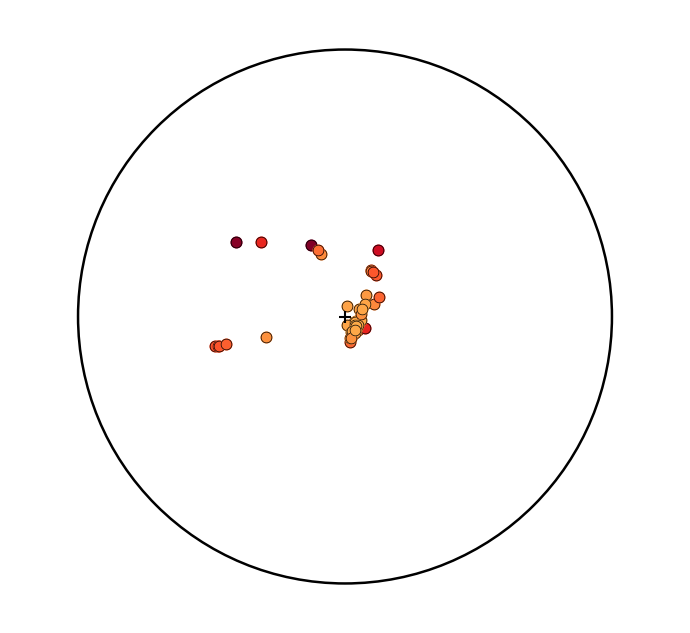}
\caption{$\bm{\hat{\theta}}$ space for Conjectures A.1 to A.78}
\label{fig:new3}
\end{subfigure}
\hfill
\caption{Projection of the generated conjecture cluster in the estimated $(\hat{\theta})$-space. This visualisation includes only the conjectures generated by HypothesiX and excludes the known reference conjectures from the display. For each generated conjecture $(c_{\mathrm{new}})$, the estimated structural feature vector $\hat{\theta}(c_{\mathrm{new}})$ is computed using the softmax-weighted interpolation procedure defined in Equation~\eqref{thetahatdefinition}, combining semantic text embeddings with structural motif similarity. The resulting six-dimensional vectors are projected onto the first two principal components of the learned feature space for visualization. Each point represents a generated conjecture, coloured according to its estimated non-triviality score $(\hat{\Upsilon})$: dark orange corresponds to conjectures with lower scores $(\hat{\Upsilon}\to 0)$, while light orange corresponds to conjectures with higher scores $(\hat{\Upsilon}\to 1)$. The figure illustrates the internal geometric organisation of the machine-generated conjecture space and highlights how the generated results cluster according to their inferred structural hardness and similarity in $(\hat{\theta})$-space.}
\label{newcluster}
\end{figure}

\medskip
\noindent
{\it Acknowledgement.} We would like to thank Yang Hui-He, Challenger Mishra, and Nicolas Robles for many helpful discussions. The author is supported by the EPSRC postdoctoral fellowship \#RE 131462.

\medskip
\noindent
{\it Rights Retention.} For the purpose of open access, the author has applied a Creative Commons Attribution (CC BY) licence to any Author Accepted Manuscript version arising from
this submission.

\appendix

\section{On the Benchmark}\label{appendix}
\noindent

\subsection{Benchmark comparison between two conjectures}
\noindent
In Theorem~\ref{pi2bq}, we prove a weaker version of Conjecture~\ref{a1}, and in Section~\ref{d2importance}, we assume that the conjecture may not universally hold for all $x\geq 7$ and square-free $Q$ such that $6 \mid Q$. However, as mentioned earlier, we have numerically verified the conjecture for $Q=30,210$ with $7\leq x\leq 10^6$, and in both cases the conjecture holds true\footnote{We have checked this using python, and the files are available on github.}.

\medskip
\noindent
We have explained scoring for Conjecture~\ref{a1} in $\bm{\theta}$-space in Section~\ref{benchmarkexample}. The purpose of this discussion is to compare the scoring system with another conjecture. We pick Conjecture A.8 (or Theorem~\ref{primedrop}) from the same conversation and comment on its score with respect to Conjecture~\ref{a1}. The benchmark assigns the score
\begin{align*}
\hat{\bm{\theta}}(c_{\mathrm{A.8}})
=[6.66, 7.06, 6.16, 8.54, 5.11, 4.29].
\end{align*}
\noindent
Here the score for $\theta_4(c_{\mathrm{A.8}})=8.54$ is higher compared to $\theta_4(c_{\mathrm{A.1}})=5.10$, since it is indeed much easier to verify the inequality appearing in Theorem~\ref{primedrop} in contrast with Conjecture~\ref{a1}. Now,
\begin{align*}
\theta_2(c_{\mathrm{A.8}})=7.06<\theta_2(c_{\mathrm{A.1}})=8.16,
\end{align*}
which is indeed the correct score, as evident from the fact that Conjecture A.8 is proven and stated as Theorem~\ref{primedrop}, whereas we do not currently have the techniques to prove Conjecture~\ref{a1}. However, we can prove a weaker version of it, namely Theorem~\ref{pi2bq}. Therefore, the scoring difference of $1.1$ is reasonable, as it lies on a scale of $[1,10]$, and the benchmark only has information from the statements of the conjectures, not their relevant proof methodologies.

\medskip
\noindent
Observe that
\begin{align*}
\theta_6(c_{\mathrm{A.8}})=4.29<\theta_6(c_{\mathrm{A.1}})=5.98,
\end{align*}
with a difference of $1.69$, as expected, since Conjecture~\ref{a1} can be used to rewrite the well-known parity problem, whereas Theorem~\ref{primedrop} does not have such a strong connection to any known open problems. We now turn our attention to
\begin{align*}
\theta_5(c_{\mathrm{A.8}})=5.11<\theta_5(c_{\mathrm{A.1}})=6.07,
\end{align*}
with a difference of $0.96$. We use the same example of the parity problem here, as the problem has been rewritten using the key idea of Conjecture~\ref{a1}, but to make progress in this direction with the proposed hybrid sieve, it is evident that we will require Theorem~\ref{primedrop}. Thus, their small difference in $\theta_5$ is justified.

Given that the definition of $B_Q(x)$ is truly novel, we can only speculate on the justification for the relation
\begin{align*}
\theta_3(c_{\mathrm{A.8}})=6.16<\theta_3(c_{\mathrm{A.1}})=7.0,
\end{align*}
based on the closest known conjectures appearing in the benchmark. Since Conjecture~\ref{a1} is closer to the Elliott–Halberstam conjecture, whereas Theorem~\ref{primedrop} is closer to the Shanks conjecture, which is weaker than the Elliott–Halberstam conjecture, it is reasonable to agree with the aforementioned inequality.

\medskip
\noindent
Finally, we note that
\begin{align*}
\theta_1(c_{\mathrm{A.8}})=6.16<\theta_1(c_{\mathrm{A.1}})=8.01,
\end{align*}
with the largest difference of $1.85$ compared to the other tuples. Both statements in the discussion play a crucial role in the development of the hybrid sieve argument presented in Section~\ref{hybridsieve}. However, even a weaker version of Conjecture~\ref{a1} is the key to defining the sieve weight and establishing a relation with the Maynard–Tao sieve weight as given in~\eqref{sieveweightconnection}. Furthermore, as the main purpose of proposing Problems 1–4 is to make progress towards the parity problems. Therefore, the consequences of Conjecture~\ref{a1} are much more significant than that of Theorem~\ref{primedrop}, justifying their scoring for $\theta_1$.


\printbibliography

\end{document}